\documentclass[amssymb,amsfonts,reqno]{amsart}

\usepackage{amssymb,amsmath,amsthm,verbatim,IEEEtrantools, mathtools}

\usepackage{latexsym}

\usepackage{graphicx}

\usepackage{bbm}

\def\be#1{\begin{equation}\label{#1}}
\def\bas{\begin{align*}}
\def\eas{\end{align*}}
\def\bi{\begin{itemize}}
\def\ei{\end{itemize}}

\usepackage{pdfpages}

\theoremstyle{plain}
   \newtheorem{theorem}[subsection]{Theorem}

   \newtheorem{proposition}[subsection]{Proposition}
   
   \newtheorem{lemma}[subsection]{Lemma}
   \newtheorem{corollary}[subsection]{Corollary}

\begin{document}

\author{James Wright}
\address{Maxwell Institute of Mathematical Sciences and the School of Mathematics, University of
Edinburgh, JCMB, The King's Buildings, Peter Guthrie Tait Road, Edinburgh, EH9 3FD, Scotland}
\email{j.r.wright@ed.ac.uk}




\title[The $H$ and $J$ functional]{The $H$ and $J$ functional in the \\
theory of complete exponential sums}

\maketitle

\begin{abstract} In this paper we introduce two nonlinear functionals, the $H$ and $J$ functional,
in the theory of complete exponential sums.
\end{abstract}

\section{Introduction}\label{intro}

Fix a prime $p$ and let $f \in {\mathbb Z}[X]$ be a polynomial with integer coefficients. In this paper
we will consider the complete exponential sum
$$
S_m(f) \ = \ p^{-m} \, \sum_{x=0}^{p^m -1} e^{2\pi i f(x)/p^m}.
$$
Let $|\cdot|$ denote the $p$-adic absolute value defined on integers $x \in {\mathbb Z}$ by $|x| = p^{-s}$
where $p^s \, | x$ but $p^{s+1} \not| \, x$. For $f\in {\mathbb Z}[X]$, we define
$$
H_f \ = \ \inf_{x \in {\mathbb Z}} \, H_f(x) \ \ \ {\rm where} \ \ \ H_f(x) \ = \  \max_{k\ge 1} \, \bigl(|f^{(k)}(x)/k!|^{1/k}\bigr)
$$
with respect to the $p$-adic absolute value $|\cdot|$. We will also use a closely associated functional
$J_f = \inf_{x\in {\mathbb Z}} J_f(x)$ where 
$J_f(x) = \max_{k\ge 2} (|f^{(k)}(x)/k!|^{1/k})$ 
is defined in the same way as $H_f(x)$ except the
maximun is taken over $k\ge 2$, using only derivatives of second order and higher.

The $J$ functional does not see the linear part of $f$. If $f_b(x) = f(x) - b x$, we note that $J_{f_b} = J_f$.

\begin{theorem}\label{exp-sum-main} Let $f \in {\mathbb Z}[X]$ have degree $d\ge 1$ and suppose $p>d$. Then
\begin{equation}\label{exp-sum-bound}
|S_m(f)| \ \le \ C_d \, \min(1,  J_{p^{-m}f}^{-1})
\end{equation}
holds for a constant $C_d$ depending only on the degree $d$. 
Furthermore there are integers $a, c \in {\mathbb Z}$, depending on $f$, such that
\begin{equation}\label{exp-sum-below}
(0.25) \, p^{-1} \,  \min(1, J_{a p^{-m} f}^{-1}) \ < \ |S_m(a f_c)| \ \le \ C_d \ \min(1, p^{-m} J_{a p^{-m} f}^{-1})
\end{equation}
where $f_c(x) = f(x) - c x$. When the $p$-adic valuation\footnote{For a nonzero
$w \in {\mathbb Z}$, the $p$-adic valuation ${\rm ord}_p(w)$ is defined by $|w| = p^{-{\rm ord}_p(w)}$.} of $J_{a f}$ is an integer, the lower bound improves to 
$(0.25) \min(1, J_{a p^{-m} f}^{-1}) \le |S_m(a f_c)|$.
\end{theorem}

If $d=1$, then $f(z) = b z$ for some $b\in {\mathbb Z}\setminus\{0\}$.
Taking $c = b$ and $a=1$, we see $f_c \equiv 0$ and so $S_m(f_c) = 1 = \min (1, J_{p^{-m}f}^{-1})$
since $J_{p^{-m} f} = 0$, illustrating \eqref{exp-sum-below} in this case.

If $d=2$ and $p>2$, then $f(z) = b z^2 + e z$ for some $b, e \in {\mathbb Z}$ with $b\not= 0$. If $|b| = p^{-k}$, then 
we have $|S_m(f)| = p^{-(m-k)/2}$ if $k\le m$ and $|S_m(f)| = 1$ if $m\le k$.
Also $J_{p^{-m} f} = p^{(m-k)/2}$ and so 
$$
\min(1, J_{p^{-m} f}^{-1})  \ = \ \begin{cases} p^{-(m-k)/2}, & {\rm if} \ k \le m \\
1, & {\rm if} \ m\le k
\end{cases},
$$
illustrating \eqref{exp-sum-below} with $a=1$ and $c=0$ since we have $|S_m(f)| = \min(1, J_{p^{-m}f}^{-1})$.

There is a stronger, sharper upper bound than \eqref{exp-sum-bound} involving the $H$ functional. 

We will formulate this stronger bound in greater generality.
The exponential sum $S_m(f)$ can be lifted to an oscillatory integral on the $p$-adic field (see Section \ref{notation}).
Let ${\rm e}$ be the standard nonprincipal character on the $p$-adic field ${\mathbb Q}_p$ where
${\rm e} \equiv 1$ on ${\mathbb Z}_p = \{z \in {\mathbb Q}_p: |z| \le 1\}$.\footnote{The $p$-adic absolute
value $|\cdot|$ has a unique extension to ${\mathbb Q}_p$.} We have 
$$
S_m(f) \ = \ \int_{{\mathbb Z}_p} {\rm e}(p^{-m} f(z)) \, dz.
$$
From our perspective, it is better to consider more
general, {\it local} exponential sums and/or oscillatory integrals: for $P \in {\mathbb Q}_p[X]$ and $H>0$, let
$$
I_P(H) \ = \ \int_{\{z\in {\mathbb Z}_p:  H \le H_P(z)\}} {\rm e}(P(z)) \, dz
$$
where $H_P(z) = \max_{k\ge 1} (|P^{(k)}(z)/k!|^{1/k})$. Of course if $H = H_f$
where $H_f = \inf_{z\in {\mathbb Z}_p} H_f(z)$, then 
$$
I_P \ := \ I_P(H_f) \ = \ \int_{{\mathbb Z}_p} {\rm e} (P(z)) \, dz.
$$
We note $S_m(f) = I_{p^{-m}f}$.

\begin{theorem}\label{H-main} For $P \in {\mathbb Q}_p[X]$, we have
\begin{equation}\label{I-single-refined}
I_P(H) \ = \ \epsilon \, p^{-s} \,  \sum_{t\in {\mathbb Z}/p{\mathbb Z}} e^{2\pi i Q(t)/p} \ + \ O_d(H^{-1})
\end{equation}
where $\epsilon \in \{0,1\}$. If $\epsilon = 1$, then $p^s = H = H_P(z_{*}) = |P'(z_{*})|$ and
$p^{s-1/2} = J_P(z_{*}) = |P''(z_{*})/2|^{1/2}$ for some 
$z_{*}\in {\mathbb Z}_p$. Furthermore 
\begin{equation}\label{P-derivatives}
|P^{(k)}(z_{*})/k!| \ \le \ p^{ks - (k-1)} \ \ {\rm for} \ \ 1 \ \le k \le d.
\end{equation}
Also $Q \in {\mathbb Z}[X]$ with
some coefficient of $Q$ not equal to 0 mod $p$. As a consequence, the Weil bound implies
\begin{equation}\label{I-single}
|I_P(H)| \ \le \ C_{d} \, \min(\sqrt{p} \, H^{-1}, J_P^{-1})
\end{equation}
if $p > d$.
\end{theorem}

In Section \ref{H-sharpness}, we will exhibit an example where $\epsilon = 1$ in \eqref{I-single-refined} occurs.

For $P \in {\mathbb Q}_p[X]$ of degree $d$, let
$P'(z) = a \prod_{\xi} (z - \xi)^{e_{\xi}}$ where $\{\xi\}$ are the distinct roots of $P'$ lying
in ${\mathbb Q}_p^{alg}$, the algebraic closure of ${\mathbb Q}_p$. The $p$-adic absolute
value $|\cdot|$ extends uniquely to ${\mathbb Q}_p^{alg}$ which we continue to denote
as $|\cdot|$. 
Theorem \ref{H-main} implies the following discrete version of a stable
bound for euclidean oscillatory integrals due to Phong and Stein \cite{PS-I}.

\begin{proposition}\label{cor-phong-stein} For $p>d \ge 2$, we have
\begin{equation}\label{phong-stein}
|I_P| \ \le \ C_d \ \max_{\xi} \min_{{\mathcal C} \ni \xi} 
\Bigl[\frac{1}{|a \prod_{\eta \notin {\mathcal C}} (\xi - \eta)^{e_{\eta}}|}\Bigr]^{\frac{1}{S({\mathcal C})+1}} 
\end{equation}
where the minimum is taken over all subsets ${\mathcal C}\subseteq \{\eta\}$ of the roots of $P'$ containing $\xi$
and $S({\mathcal C}) = \sum_{\eta \in {\mathcal C}} e_{\eta}$.
\end{proposition}

In Section \ref{PS-cor}, we will show how Theorem \ref{H-main} implies Proposition \ref{cor-phong-stein}.
Previously \eqref{phong-stein} was established
with an additional factor of $p^2$. See \cite{W-MRL}.
Applying the bound \eqref{phong-stein} with singletons ${\mathcal C} = \{\xi\}$ 
immediately implies a celebrated bound due to Loxton and Vaughan for the complete
exponential sums $S_m$.

\begin{corollary}\label{loxton-vaughan} \cite{LV} For $p>d \ge 2$, we have
\begin{equation}\label{LV-1}
|S_m(f)| \ \le \ C_d \, \max_{\xi} 
\Bigl[\frac{p^{-m}}{|a \prod_{\eta\not= \xi} (\xi - \eta)^{e_{\eta}}|}\Bigr]^{1/(e_{\xi}+1)}.
\end{equation}
\end{corollary}
Strictly speaking Theorem 1 in \cite{LV} states \eqref{LV-1} with the powers $1/(e_{\xi}+1)$ replaced
by $1/(e+1)$ where $e = \max_{\xi} e_{\xi}$. However, as the authors remark, the argument in \cite{LV} also gives
the sharper bound \eqref{LV-1}. In \cite{LV}, the constant $C_d$ in \eqref{LV-1} can be taken to be
$C_d = d-1$ whereas our argument only gives $C_d = (d-1)d/2$. 

As the coefficients of the polynomial $P(x) = c_1 x + \cdots c_d x^d \in {\mathbb Q}_p[X]$ get large, say 
$\max_j |c_j| \to \infty$ (this corresponds to $m\to\infty$ for the complete exponential sums $S_m$),
then it is the singletons ${\mathcal C} = \{\xi\}$ which govern the behaviour and the Loxton-Vaughan
bound becomes very sharp. On the other hand, if the coefficients stay bounded and the roots of $P'$ start
to cluster together, then larger collections of roots ${\mathcal C}$ 
govern the behaviour of the exponential sum. The advantage of the bound
\eqref{phong-stein} is that it is stable under perturbations of the polynomial phase $P$. This is what makes
the Phong-Stein bound for euclidean oscillatory integrals 
very useful.

As an example, consider $f(x) = (x - n_1)^{f} (x-n_2)^{e}(x-n_3)^{e} \in {\mathbb Z}[X]$ with integer
roots $n_1, n_2, n_3 \in {\mathbb Z}$ where $|n_1 - n_2| = |n_1 - n_3| = p^{-s}, \ |n_2 - n_3|=p^{-t}$
and $t \gg s$. Then the bound \eqref{phong-stein} implies
$$
|S_m(f)| \ \lesssim \ \min\bigl( p^{[-m + f s + e t]/(1+e)}, \, p^{[-m+f s]/(1+2e)} \bigr).
$$
The singletons ${\mathcal C} = \{\xi\}$ govern the behaviour in the regime $f s + (1+2e) t \le m$ but when
$n_2 \to n_3$ or $|n_2 - n_3| = p^{-t} \to 0$ as $m$ stays bounded, then the two roots $\{n_2, n_3\}$, each
with multiplicity $e$ should be considered as a single root with multiplicity $2e$ and the estimate above
remains stable as we make a transition from a bound for singletons to a bound corresponding to a larger cluster of roots.

\subsection*{Polynomials in many variables} 

An advantage of our methods is that it can treat local objects such as $I_P(H)$ and this allows us (almost for free)
to treat exponential sums for polynomials of many variables.

Let $P \in {\mathbb Q}_p[X_1,\ldots, X_n]$ be a polynomial of degree $d$ in $n$ variables
with coefficients in the $p$-adic field ${\mathbb Q}_p$; that is (notation will be defined in Section \ref{notation}),
$$
P({\underline{z}}) \ = \ \sum_{|\alpha| \le d} c_{\alpha} \, {\underline{z}}^{\alpha} \ \ \ {\rm where} \ \ \
c_{\alpha} \in {\mathbb Q}_p \ \ {\rm for \ all} \ \alpha \in {\mathbb N}_0^n.
$$
For ${\underline{z}} \in {\mathbb Z}_p^n$, set
$$
H({\underline{z}}) \ = \ \max_{|\alpha|\ge 1} \  
\Bigl| \frac{\partial^{\alpha}P({\underline{z}})}{\alpha!} \Bigr|^{1/|\alpha|}.
$$
We consider the oscillatory integral
$$
I_P(H) \ = \ \int_{\bigl\{{\underline{z}} \in {\mathbb Z}_p^n : \, H \le H({\underline{z}})\bigr\}} 
{\rm e}(P({\underline{z}})) \, d{\underline{z}} .
$$

\vskip 10pt

\begin{theorem}\label{main-bound}
We have
\begin{equation}\label{I-several}
|I_P(H)| \ \le \ C_{d,n} \ p \, H^{-1}.
\end{equation}
\end{theorem}

{\bf Remark}: 
Of course if we set $H_P = \inf_{{\underline{z}} \in {\mathbb Z}_p^n} H({\underline{z}})$ and
$$
I_P \ = \ \int_{{\mathbb Z}_p^n} {\rm e}(P({\underline{z}})) \, d{\underline{z}},
$$
then \eqref{I-several} implies $|I_P| \le C_{d,n} \, p \, H_P^{-1}$.

Let $Q \in {\mathbb Z}[X_1,\ldots, X_n]$ be a polynomial of degree $d$ in $n$ variables with
integer coefficients. 
The argument we give to prove Theorem \ref{main-bound} gives a bound on the number of solutions
to the polynomial congruence
\begin{equation}\label{poly-congruence}
Q(x_1, \ldots, x_n) \ \equiv \ a \ \ {\rm mod} \ \ p^{\alpha}.
\end{equation}
Set
$P_{\alpha}({\underline{x}}) = p^{-\alpha} Q({\underline{x}})$ and ${\mathcal H}_{\alpha} = H_{P_{\alpha}}$.
If $N_{a,\alpha}$ denotes the number of solutions to \eqref{poly-congruence}, 
then we have the following bound.

\begin{proposition}\label{congruence-bound} We have
\begin{equation}\label{congruence-bound-several}
\sup_{a\in {\mathbb Z}}  p^{-n\alpha} N_{a, \alpha} \ \le \ C_{d,n} \, 
\min(1, {\mathcal H}_{\alpha}^{-1}).
\end{equation}
When $n=1$, we also have a general lower bound:
\begin{equation}\label{congruence-n=1}
p^{-1} \min(1, {\mathcal H}_{\alpha}^{-1}) \ < \ \sup_{a \in {\mathbb Z}} \,
p^{-\alpha} \, N_{a, \alpha} \ \le \ C_d \, \min(1, {\mathcal H}_{\alpha}^{-1})
\end{equation}
with an improvement to 
$\min(1, {\mathcal H}_{\alpha}^{-1}) \le \sup_{a \in {\mathbb Z}} p^{-\alpha} N_{a,\alpha} \le C_d {\mathcal H}_{\alpha}^{-1}$
when the $p$-adic valuation of ${\mathcal H}_{\alpha}$ is an integer.
\end{proposition}


\subsection*{Structure of the paper}  In the next two sections we set some notation, review basic
analysis on the $p$-adic field and establish some preliminary results. In Section \ref{lower}, we establish
the lower bound \eqref{congruence-n=1} in Proposition \ref{congruence-bound}. In Section \ref{H-sharpness},
we make comments on the sharpness of Theorem \ref{H-main} and in the following section, we establish the lower bound
\eqref{exp-sum-below} in Theorem \ref{exp-sum-main}. In Section \ref{PS-cor}, we show how Theorem \ref{H-main}
implies Proposition \ref{cor-phong-stein}. In Section \ref{structure}, we give the main structural sublevel set statement
which lies at the heart of our analysis but we postpone the proof until Sections \ref{Prop-structure} and 
\ref{hensel-proof}. In Section \ref{upper}, we complete the proof of Proposition \ref{congruence-bound}
by establishing the upper bound. In Section \ref{1.5-first}, we give the proof of Theorem \ref{main-bound}
and in Sections \ref{H-first}, \ref{H-final}, \ref{establishing} and \ref{push}, we give the proof
of Theorem \ref{H-main}.

\subsection*{Notation}
We use the notation $A \lesssim B$ between two positive quantities $A$ and $B$ to denote $A \le CB$ for some
constant $C$. We sometimes use the notation $A \lesssim_k B$ to emphasise that the
implicit constant depends on the parameter $k$. We sometimes use $A = O(B)$ to denote the
inequality $A \lesssim B$.
Furthermore, we use $A \ll B$ to denote $A \le \delta B$ for a sufficiently small constant $\delta>0$
whose smallness will depend on the context.

\subsection*{Acknowlegement}
We thank Jonathan Hickman and Rob Fraser for clarifying discussions on topics related to this paper.

\section{A review of basic analysis over ${\mathbb Q}_p$}\label{notation}

A standard reference for $p$-adic analysis is \cite{K}.

The $p$-adic absolute value $|\cdot|$ extends from the integers ${\mathbb Z}$ to the rationals
${\mathbb Q}$ to the $p$-adic field ${\mathbb Q}_p$ in a unique way. General elements
$z \in {\mathbb Q}_p$ have a unique expansion
$$
z \ = \ \frac{a_{-M}}{p^M} + \cdots + \frac{a_{-1}}{p} \ + \ a_0  \ + \ a_1 p + a_2 p^2 \ + \cdots \ = \ L(z) \ + \
\sum_{j\ge 0} a_j p^j
$$
where $a_k \in {\mathbb Z}/p{\mathbb Z}$. The $p$-adic absolute value of $z$ is
then $|z| = p^M$ where $p^{-M}$ is the smallest power of $p$ in the expansion
with a nonzero coefficient. Hence $w\in {\mathbb Z}_p = \{z \in {\mathbb Q}_p : |z|\le 1\}$ precisely when $L(w) = 0$

The basic nonprincipal character ${\rm e}$ on ${\mathbb Q}_p$ is defined
by ${\rm e}(z) = e^{2\pi i L(z)}$ so that ${\rm e} \equiv 1$ on ${\mathbb Z}_p$ and ${\rm e}$ is nontrivial
$\{z \in {\mathbb Q}_p: |z| \le p\}$. Let $B_r(z_0) = \{z \in {\mathbb Q}_p: |z - z_0| \le r\}$ denote the ball
with centre $z_0$ and radius $r>0$ defined with respect to the $p$-adic absolute value.

For a polynomial $f \in {\mathbb Z}[X]$ and $m\in {\mathbb N}$, we have
$$
\int_{{\mathbb Z}_p} {\rm e}(p^{-m} f(z)) \, dz \ = \ \sum_{x=0}^{p^m -1} \int_{B_{p^{-m}}(x)} {\rm e}(p^{-m} f(z)) \, dz .
$$
Since $f(z) = f(x)$ mod $p^m$ when $z \in B_{p^{-m}}(x)$, 
we have ${\rm e}(p^{-m}f(z)) = {\rm e}(p^{-m}f(x))$ and also 
${\rm e}(p^{-m} f(x)) = e^{2\pi i f(x)/p^{m}}$ since $f(x) \in {\mathbb Z}$. Therefore
$$
\int_{{\mathbb Z}_p} {\rm e}(p^{-m} f(z)) \, dz \ = \ \sum_{x=0}^{p^m -1} {\rm e}(p^{-m}f(x)) |B_{p^{-m}}(x)| \ = \
p^{-m} \sum_{x=0}^{p^m -1} e^{2\pi i f(x)/p^m}
$$
which shows that complete exponential sums $S_m(f)$,
$$
\sum_{x=0}^{p^m-1} e^{2\pi i f(x)/p^m} \ = \ \int_{{\mathbb Z}_p} {\rm e}(p^{-m}Q(z)) \, dz,
$$
can be written as a $p$-adic oscillatory integral. In the same way, a normalised count $p^{-n\alpha} N_{a,\alpha}$
of the number of solutions to the congruence $Q(x_1, \ldots, x_n) \equiv a$ mod $p^{\alpha}$ where
$Q \in {\mathbb Q}_p[X_1, \ldots, X_n]$ can be written as the measure of the sublevel set
$$
p^{-n\alpha} N_{a,\alpha} \ = \ \bigl| \{{\underline{z}} \in {\mathbb Z}_p^n : |Q({\underline{z}}) - a| \le p^{-\alpha} \}\bigr|.
$$

We will consider polynomials $P \in {\mathbb Q}_p[X_1,\ldots, X_n]$ of many variables so that
$$
P({\underline{z}}) \ = \ \sum_{|\alpha| \le d} c_{\alpha} \,  {\underline{z}}^{\alpha}
$$
where $\alpha = (\alpha_1, \ldots, \alpha_n) \in {\mathbb N}_0^n$ is a multi-index, 
$|\alpha| = \alpha_1 + \cdots + \alpha_n$, each
$c_{\alpha} \in {\mathbb Q}_p$ and for ${\underline{z}} = (z_1,\ldots, z_n)$,
we have ${\underline{z}}^{\alpha} = z_1^{\alpha_1}\cdots z_n^{\alpha_n}$. Here ${\mathbb N}_0 = \{0,1,2,\ldots \}$.

We extend the $p$-adic absolute value $|\cdot|$ on ${\mathbb Q}_p$ to a norm on
vectors ${\underline{w}} = (w_1, \ldots, w_n) \in {\mathbb Q}^n_p$. We set
$|{\underline{w}}| = \max_j |w_j|$. We denote by 
$$
B_r({\underline{w}}) \ = \ \bigl\{{\underline{z}} \in {\mathbb Q}_p^n : \, |{\underline{z}} - {\underline{w}}| \le r \bigr\}
$$
the $p$-adic ball of radius $r$ and centre ${\underline{w}} \in {\mathbb Q}_p^n$.

The usual rules of calculus apply to polynomials over the $p$-adics. For example, we have
$$
\lim_{h\to 0} \frac{P(z_1,\cdots, z_j+h, \cdots, z_n) - P(z_1,\cdots, z_n)}{h} \ = \
\sum_{|\alpha|\le d} \alpha_j c_{\alpha} z_1^{\alpha_1}\cdots z_j^{\alpha_j -1} \cdots z_n^{\alpha},
$$
the limit exists in ${\mathbb Q}_p$. We denote this limit by $(\partial/\partial z_j) P({\underline{z}})$.
More generally we define the partial derivative
$$
\partial^{\alpha} P  \ := \ 
\frac{\partial^{|\alpha|} P}{\partial^{\alpha_1}z_1 \cdots \partial^{\alpha_n} z_n} 
$$
as a composition of the first order partial derivatives defined above. One can easily verify the
Taylor expansion for polynomials
$$
P({\underline{z}}) \ = \ \sum_{\alpha} \bigl[\partial^{\alpha} P({\underline{w}})/\alpha!\bigr] \, 
({\underline{z}} - {\underline{w}})^{\alpha} \ \ {\rm where} \ \ \alpha! = \alpha_1! \cdots \alpha_n!
$$
centred at any point ${\underline{w}}$.

Important to us will be the
differential operators $({\underline{u}}\cdot \nabla)^k$ defined as the $k$-fold
composition of the directional derivatives
$$
{\underline{u}}\cdot \nabla \ = \ u_1 \frac{\partial}{\partial z_1} + \cdots + u_n 
\frac{\partial}{\partial z_n}
$$
where ${\underline{u}} \in {\mathbb Q}_p^n$ defines the direction. One checks that
$({\underline{u}}\cdot \nabla)^k P({\underline{z}}) = g^{(k)}(0)$ where
$g(t) = P({\underline{z}} + t {\underline{u}}) \in {\mathbb Q}_p[X]$, with coefficients depending on
${\underline{z}}, {\underline{u}} \in {\mathbb Q}_p^n$.

\section{The space of homogeneous polynomials}\label{homogeneous}

Let $V_{n,k}$ denote the vector space of homogeneous polynomials in ${\mathbb Q}_p[X_1, \ldots, X_n]$
of degree $k$. So $Q \in V_{n,k}$ means 
$Q({\underline{z}}) = \sum_{|\alpha| = k} b_{\alpha} \, {\underline{z}}^{\alpha}$.

\begin{lemma}\label{V-basis}
Let $d(n,k)$ denote the dimenson of the vector space $V_{n,k}$. There exists a sequence
of unit vectors ${\underline{u}}_j, j=1,\ldots, d(n,k)$ in ${\mathbb Q}_p^n$ such that 
$$
Q_j({\underline{z}})\  := \  ({\underline{u}}_j \cdot {\underline{z}} )^k, \ \ j=1, \ldots, d(n,k)
$$
forms a basis for $V_{n,k}$.
\end{lemma}

\begin{proof} It suffices to show that ${\rm Span}\bigl\{ ({\underline{u}} \cdot {\underline{z}})^k : {\underline{u}}
\in {\mathbb Q}_p^n \bigr\} = V_{n,k}$. Suppose not. Then there exists a nonzero linear
functional $\Lambda : V_{n,k} \to {\mathbb Q}_p$ which vanishes on $M = {\rm Span}\bigl\{ ({\underline{u}} \cdot {\underline{z}})^k : {\underline{u}}
\in {\mathbb Q}_p^n \bigr\}$. For every $\alpha \in {\mathbb N}_0^n$ with $|\alpha| = k$, set
$c_{\alpha} = \Lambda({\underline{z}}^{\alpha})$ and since $\Lambda $ is nonzero and
$\{{\underline{z}}^{\alpha}\}$ forms a basis for $V_{n,k}$, the coefficients $\{c_{\alpha}\}$ are
not all equal to zero. 

For any ${\underline{u}} \in {\mathbb Q}_p^n$, 
$$
({\underline{u}}\cdot {\underline{z}})^k \ = \ \sum_{|\alpha| =k} b_{k,\alpha}\,  (u_1 z_1)^{\alpha_1} \cdots
(u_n z_n)^{\alpha_n}
$$
where $b_{k,\alpha} = k!/(\alpha_1 ! \cdots \alpha_n !)$. Hence
$$
0 \ = \ \Lambda(({\underline{u}}\cdot {\underline{z}})^k) \ = \ \sum_{|\alpha|=k} b_{k,\alpha} c_{\alpha} \,
{\underline{u}}^{\alpha}
$$
where the polynomial 
$$
P({\underline{w}}) \ := \ \sum_{|\alpha|=k} b_{k,\alpha} c_{\alpha} {\underline{w}}^{\alpha} \ \in \ {\mathbb Q}_p[X_1, \ldots, X_n]
$$
is a nonconstant polynomial. A standard fact from algebra states that for any nonconstant polyonomial
$P \in k[X_1,\ldots, X_n]$
over an infinite field $k$, there exists a ${\underline{u}} \in k^n$ such that $P({\underline{u}}) \not= 0$. 

This is a contradiction.  
\end{proof}

As a consequence of Lemma \ref{V-basis}, we see that for every $\alpha \in {\mathbb N}_0^n$
with $|\alpha| = k$, 
$$
{\underline{z}}^{\alpha} \ = \ \sum_{j=1}^{d(n,k)} c_j ({\underline{u}}_j \cdot {\underline{z}})^k
$$
for some choice of coefficients $c_j = c_j(\alpha) \in {\mathbb Q}_p$. Hence we can write
\begin{equation}\label{partial-u}
\partial^{\alpha} \ = \ \sum_{j=1}^{d(n,k)} c_j \, ({\underline{u}}_j \cdot \nabla)^k
\end{equation}
for every $\alpha$ with $|\alpha| = k$.

We note that the coefficients $\{c_j(\alpha)\}$
appearing in \eqref{partial-u} are universal, depending only on $d$ and $n$ and not depending on
the polynomial $P$. If $\{q^t\}$ denote all the prime powers arising as divisors of the coefficients
$c_j(\alpha)$, then $q^t \le C_{d,n}$.

\section{The lower bound \eqref{congruence-n=1} in Proposition \ref{congruence-bound}}\label{lower}

Let $Q \in {\mathbb Z}[X_1, \ldots, X_n]$ and let
$N_{a,\alpha}$ denote the number of solutions to $Q \equiv a$ mod $p^{\alpha}$; see \eqref{poly-congruence}. We note that
the normalised count
$$
p^{-n\alpha} N_{a,\alpha} \ = \ 
\bigl|\{ {\underline{z}} \in {\mathbb Z}_p^n : \, |Q({\underline{z}}) - a| \le p^{-\alpha} \}\bigr|
$$
is given as the $p$-adic measure of a sublevel set for $Q - a$. Hence
$$
\sup_{a \in {\mathbb Z}} p^{-n\alpha} N_{a, \alpha} \ = \ \sup_{a\in {\mathbb Z}_p} \,
\bigl|\{ {\underline{z}} \in {\mathbb Z}_p^n : \, |Q({\underline{z}}) - a| \le p^{-\alpha} \}\bigr|.
$$

Recall that $P_{\alpha}({\underline{z}}) =
p^{-\alpha} Q({\underline{z}})$ and ${\mathcal H}_{\alpha} = H_{P_{\alpha}}$. Consider any
${\underline{z}}_{*} \in {\mathbb Z}_p^n$ where ${\mathcal H}_{\alpha} = H_{P_{\alpha}}({\underline{z}}_{*})$
and let $a_{*} = Q({\underline{z}}_{*}) \in {\mathbb Z}_p$. We have
$$
S \ := \ \sup_{a\in {\mathbb Z}}
p^{-n\alpha} N_{a,\alpha} \ \ge \  
\bigl|\{ {\underline{z}} \in {\mathbb Z}_p^n : \, |P_{\alpha}({\underline{z}}) - P_{\alpha}({\underline{z}}_{*})| \le 1 \}\bigr|
$$
and so
$$
S \ \ge \ 
\bigl| \{ {\underline{z}}\in B_{{\mathcal H}_{\alpha}^{-1}}({\underline{z}}_{*}) \cap {\mathbb Z}_p^n :  
|\sum_{|\beta| \ge 1} [\partial^{\beta} P({\underline{z}}_{*})/\beta!] 
({\underline{z}} - {\underline{z}}_{*})^{\beta} | \le 1 \} \bigr| \
= \ |B_{{\mathcal H}_{\alpha}^{-1}}({\underline{z}}_{*}) \cap {\mathbb Z}_p^n|
$$
since $|\partial^{\beta}P({\underline{z}}_{*})/\beta!| \le {\mathcal H}_{\alpha}^{|\beta|}$ 
for each $|\beta|\ge 1$. 

When $n=1$, we have 
$$
|B_{{\mathcal H}_{\alpha}^{-1}}({\underline{z}}_{*}) \cap {\mathbb Z}_p| \ > \ p^{-1} \min(1, {\mathcal H}_{\alpha}^{-1})
$$
with an improvement to 
$|B_{{\mathcal H}_{\alpha}^{-1}}({\underline{z}}_{*}) \cap {\mathbb Z}_p| = \min(1, {\mathcal H}_{\alpha}^{-1})$
when the $p$-adic valuation of ${\mathcal H}_{\alpha}$ is an integer. This establishes the lower bound 
\eqref{congruence-n=1} in
Proposition \ref{congruence-bound}.

\section{Comments on the sharpness of the bound in Theorem \ref{H-main}}\label{H-sharpness}

First we observe by example that $\epsilon = 1$ can occur in \eqref{I-single-refined}.
 
Consider the example $P(t) = a t^3 + b t \in {\mathbb Q}_p[X]$ where $|a| = p^{3r - 2}$ and $|b| = p^r$.
We write $b = b_0 p^{-r}$ and $a = a_0 p^{-3r+2}$ where $|b_0| = |a_0| = 1$. We choose any pair $(a_0, b_0) \in [{\mathbb Z}/p{\mathbb Z}]^2$
such that the quadratic congruence $3 a_0 s^2 + b_0 \equiv 0$ mod $p$ is not solvable.

{\bf Claim}: Under the above condition, we have $H = p^r$ where
$$
H \ = \ \inf_{t\in {\mathbb Z}_p} \max \bigl( |b + 3 a t^2|, |3 a t|^{1/2}, p^{r - 2/3} \bigr) .
$$

The size conditions on $a$ and $b$ imply $H \le p^r$. 
To see the reverse inequality, we may assume there is some cancellation in the sum $b + 3 a t^2$; that is,
$|3a t^2| = p^r$. Otherwise $|b + 3 a t^2| \ge p^r$ and we would be done. Hence we may assume that $|t| = p^{-r+1}$.
If we write $t = p^{r-1} s$ where $|s| =1$, we have
$$
| b + 3 a t^2| \ = \ p^r |b_0 + 3 a_0 s^2| \ = \ p^r
$$
since $|b_0 + 3 a_0 s^2| \le p^{-1}$ means the congruence $3 a_0 s^2 + b_0 \equiv 0$ mod $p$ is solvable.
Hence $H = p^r$ and $H = H(p^{r-1} s)$ for any $|s| = 1$. Furthermore, 
$$
J(p^{r-1}s) \ = \  
\max(|3 a p^{r-1} s|^{1/2}, |a|^{1/3}) \ = \ p^{r-1/2}  \ {\rm and}  \ 
|P'''(z)/3!| \equiv |a| = p^{3r -2}.
$$
Hence \eqref{P-derivatives} holds.

Next we have
$$
I_P \ = \ \int_{{\mathbb Z}_p} {\rm e}(P(z)) \, dz \ = \ p^{-M} \, \sum_{x=0}^{p^M - 1} e^{2\pi i [ a_0 x^3 + p^{2r-2} b_0 x]/p^M}
$$
where $M = 3r - 2$.  A calculation shows that
$$
I_P \ = \ p^{-r} \, \sum_{x=0}^{p-1} e^{2\pi i [ a_0 x^3 + b_0 x]/p} \ = \ H^{-1} \, S(a_0,b_0)
$$
where $S(a_0,b_0)$ is an exponential sum over the finite field ${\mathbb Z}/p{\mathbb Z}$. 
Hence in this case, we have $\epsilon =1$.

\subsection*{Comparing to the euclidean case}
The corresponding estimate for the euclidean oscillatory integral
$$
I \ = \ \int_0^1 e^{2\pi i P(x)} \, dx
$$
where $P \in {\mathbb R}[X]$ is
$$
|I| \ \le \ C_{d} \, H^{-1} \ \ {\rm where} \ \ H = \inf_{x \in [0,1]} H(x) \ = \ \inf_{x\in [0,1]} 
\max_{k\ge 1} \bigl(|P^{(k)}(x)/k!|^{1/k} \bigr).
$$
A proof
can be found in \cite{ACK}, including the corresponding higher dimensional analogue. 
In \cite{ACK}, sharpness of the above bound was observed in the following sense: for any polynomial $P \in {\mathbb R}[X]$, there
is a $0< c \le 1$ such that
\begin{equation}\label{real-sharp}
\min(1, H^{-1}) \ = \ \Bigr| \int_0^c e^{2\pi i P(x)} \, dx \Bigr|.
\end{equation}
This is not too hard to prove but it relies heavily on the ordered structure of the real field ${\mathbb R}$.
One can use
\eqref{real-sharp} to compare the bound $|I|\le C_d H^{-1}$ to other bounds which are robust
under truncations of the oscillatory integral.

For instance, a very sharp bound for $I$ mentioned in the Introduction is due to Phong and Stein \cite{PS-I}:
$$
\Bigl| \int_a^b e^{2\pi i P(x)} \, dx \Bigr| \ \le \ C_d \, 
\max_{\xi} \min_{\xi \in {\mathcal C}} 
\Bigl[\frac{1}{|a \prod_{\eta \notin {\mathcal C}} (\xi - \eta)^{e_{\eta}}|}\Bigr]^{\frac{1}{S({\mathcal C})+1}} 
$$
so that \eqref{phong-stein} is the exact discrete analogue but here $|\cdot|$ denotes the usual
archimedean absolute value measuring the distance between the complex roots of $P'$ instead of the $p$-adic absolute value used in Corollary \ref{cor-phong-stein}.
The estimate above holds for any $a < b$ and hence \eqref{real-sharp} implies
\begin{equation}\label{PS-H}
\min_{\xi} \max_{\xi \in {\mathcal C}}
|a \prod_{\eta \notin {\mathcal C}} (\xi - \eta)^{e_{\eta}}|^{1/(S({\mathcal C})+1)} \ \le \ C_d \, H. 
\end{equation}

In the next section, we propose an alternative version of \eqref{real-sharp} which is valid in both archimedean and
nonarchimedean settings.

\section{Proof of the lower bound \eqref{exp-sum-below} in Theorem \ref{exp-sum-main}}\label{exp-lower}

Here we establish a general lower bound for 
$$
I_P \ = \ \int_{{\mathbb Z}_p} {\rm e} (P(z)) \, d z
$$ 
where $P \in {\mathbb Q}_p[X]$. The bound \eqref{I-single} from
Theorem \ref{H-main} implies
$$
|I_P| \ \le \ C_{d} \, \min(1, J_P^{-1}) \ \ {\rm where} \ \ J_P = \inf_{z\in{\mathbb Z}_p} 
\max_{2\le k \le d} \
(|P^{(k)}(z)/k!|^{1/k}).
$$
To analyse sharpness of this bound, set 
$$
\alpha_P \ := \ \sup_{w, b} \ \bigl[\max\bigl(1, J_{wP}\bigr) |I_{wP_b}|\bigr] \ \ {\rm where} \ \ P_b(z) = P(z) - bz.
$$
Here the supremum is taken over $w \in {\mathbb Z}_p$ and $b \in {\mathbb Q}_p$.
Since the $J$ functional does not see linear terms, we see $J_{f_b} = J_f$. Hence the
above bound implies $\alpha_P \le C_{d}$.
We seek a lower bound for $\alpha_P$.

Consider
$$
S_P \ := \ \sup_{a,b \in {\mathbb Q}_p} 
\bigl| \{z\in {\mathbb Z}_p : |P(z) - a - b z| \le 1 \} \bigr|.
$$
Let $z_{*} \in {\mathbb Z}_p$ be such that 
$J_P = J_P(z_{*})$. Then for $a = P(z_{*}) + b P'(z_{*})$ 
and $b = P'(z_{*})$,
we have 
$$
S_P \ \ge \ \bigl| \{ z\in B_{J_P^{-1}}(z_{*})\cap{\mathbb Z}_p : |\sum_{k\ge 2} [P^{(k)}(z_{*})/k!] 
(z - z_{*})^{k} | \le 1 \} \bigr| \ = \ |B_{J_P^{-1}}(z_{*})\cap {\mathbb Z}_p|
$$
since $|P^{(k)}(z_{*})/k!| \le J_P^{k}$ 
for each $k\ge 2$. Hence we see that $S_P > p^{-1} \min(1, J_P^{-1})$ with an improvement to
$S_P \ge \min(1, J_P^{-1})$ if $J_P = p^m$ for some integer $m\in {\mathbb Z}$.

On the other hand, if ${\mathcal D}_P^{a,b} = \{z\in {\mathbb Z}_p: |P(z) - a - b z| \le 1 \}$, then
$$
{\mathbbm{1}}_{{\mathcal D}_P^{a,b}}(z) \ = \ {\mathbbm{1}}_{{\mathbb Z}_p}(P(z) - a - b z) \ = \ 
\int_{{\mathbb Z}_p} {\rm e} (w(P(z) - a - bz)) \, dw
$$
and hence
$$
|{\mathcal D}_P^{a,b}| \ = \ 
\int_{{\mathbb Z}_p} \Bigl[ \int_{{\mathbb Z}_p} {\rm e} (w(P(z) - a - b z)) \, dw \Bigr] \,
dz \ = \ \int_{{\mathbb Z}_p} {\rm e} (-aw) I_{wP_{b}} \, dw .
$$
Since $J_{wP_{b}} = J_{w P}$, we have
$$
p^{-1} \, \min(1, J_P^{-1}) \ < \ S_P \ \le \ \alpha_P \, \int_{{\mathbb Z}_p} \min(1, J_{w P}^{-1}) \, dw.
$$
But since $|w|^{1/2} J_P \le J_{wP}$ for $w \in {\mathbb Z}_p$, we have 
$$
p^{-1} \, \min(1, J_P^{-1}) \ < \ \alpha_P \, \min(1, J_P^{-1}) \int_{{\mathbb Z}_p} |w|^{-1/2} \, dw \ \le \ 
4 \, \alpha_P \, \min(1, J_P^{-1}),
$$
implying $(0.25) \, p^{-1} < \alpha_P$ with an improvement $1/4 \le \alpha_P$ if $J_p = p^m$ for some 
$m\in {\mathbb Z}$. 

We conclude that for every $P \in {\mathbb Q}_p[X]$, there is a $w\in {\mathbb Z}_p$ and a $b \in {\mathbb Q}_p$ such that
$$
(1/4) \, p^{-1} \, \min(1, J_{wP}^{-1}) \ < \ \Bigl| \int_{{\mathbb Z}_p} {\rm e} (wP_b(z)) \, dz \Bigr|.
$$
This is our version of \eqref{real-sharp}.

For $f\in {\mathbb Z}[X]$, we apply the above to $P(z) = p^{-m} f(z)$, tracking the value of $b$, to conclude
that there are integers $a, c$ such that
$$
(1/4) \, p^{-1} \, \min(1, p^{-m} J_{af}^{-1}) \ < \ \bigl| p^{-m} \sum_{x=0}^{p^m -1} e^{2\pi i [a f_c (x)]/p^m} \bigr|,
$$
estabishing the lower bound \eqref{exp-sum-below} in Theorem \ref{exp-sum-main}.

\section{Proof of Proposition \ref{cor-phong-stein}}\label{PS-cor}

Here we show how Theorem \ref{H-main} implies Proposition \ref{cor-phong-stein}. 

Recall that $P'(z) = a \prod_{\xi} (z - \xi)^{e_{\xi}}$. The proof of Corollary \ref{cor-phong-stein}
comes in two steps. First we suppose $\epsilon = 0$ in Theorem \ref{H-main}. Then
$|I_P| \le C_d H_P^{-1}$ and so \eqref{phong-stein} follows from the following bound: for any
$z \in {\mathbb Z}_p$, there is a root $\xi$ of $P'$ such that 
\begin{equation}\label{first-step-reduction}
\max_{{\mathcal C} \ni \xi} |a \prod_{\eta\notin {\mathcal C}} (\xi - \eta)^{e_{\eta}} |^{1/(S({\mathcal C})+1)}
 \ \le \ H_P(z) 
\end{equation}
holds. Here we do not need any special properties for $P$. The inequality \eqref{first-step-reduction}
is a general inequality which holds for any polynomial $P \in {\mathbb Q}_p[X]$.

For the second step, we suppose $\epsilon =1$ in Theorem \ref{H-main}. In this case
we have $|I_P| \le C_d \sqrt{p} \, p^{-s}$ when $p>d$ where $p^s = H = H(z_{*}) = |P'(z_{*})|$ for some
$z_{*} \in {\mathbb Z}_p$. Furthermore $p^{s-1/2} = J(z_{*}) = |P''(z_{*})/2|^{1/2}$
and \eqref{P-derivatives} holds.

\subsection*{The first step} But first we prove \eqref{first-step-reduction}.

We fix $z$
and prove \eqref{first-step-reduction} for a root
$\xi$ of $P'$ with the property
$|z - \xi| \ = \ \min_{\eta} |z - \eta|$. 
We divvy up the roots $\{\xi_j^k\}_{1\le j \le r_k, 1\le k \le t}$ of $P'$  so that
$$
|z - \xi_1^k| \ = \ |z - \xi_2^k| \ = \ \cdots \ = |z - \xi_{r_k}^k| \ =: \ A_k
$$
with $0 \le A_1 < A_2 < \ldots < A_t$. Here we have set $\xi_1^1 = \xi$. We will take special
care with the case $A_1 = 0$ where $z = \xi$ and $P'(z) = 0$.

We fix a general subset ${\mathcal C}$ of the roots of $P'$ and split it as
$$
{\mathcal C} \ = \ C_1 \cup C_2 \cup \cdots \cup C_t \ \ {\rm where} \ \ 
C_k \ = \ \{\xi_1^k, \ldots, \xi_{r_k}^k\} \cap {\mathcal C}.
$$
We set $f_k = \sum_{\eta \in C_k} e_{\eta}$ so that $S({\mathcal C}) = f_1 + \cdots + f_t$. 
When $A_1 = 0$, we have $C_1 = \{\xi\}$ and so $f_1 = e_{\xi}$.

For each $1\le k \le t$, set
$$
F_k(z) \ = \ \prod_{j=1}^{r_{k}} (z - \xi_j^k)^{e_{\xi_j^k}} \ =: \ F_k^1(z) \, F_k^2(z)
$$
where
$$
F_k^1(z) \ = \ \prod_{\xi_j^k \in C_k} (z - \xi_j^k)^{e_{\xi_j^k}} \ \ {\rm and} \ \
F_k^2(z) \ = \ \prod_{\xi_j^k \notin C_k} (z - \xi_j^k)^{e_{\xi_j^k}} .
$$
Therefore $P'(z) = a \prod_{k=1}^t F_k(z)$ and if 
$Q_{{\mathcal C}} := | a \prod_{k=1}^t F_k^2(\xi)|$, then our goal is to prove
\begin{equation}\label{Q-aim}
Q_{{\mathcal C}}^{1/(S({\mathcal C}) +1)} \ \le \ H_P(z)
\end{equation}
which will establish \eqref{first-step-reduction} since ${\mathcal C}$ is a general subcollection
of the roots of $P'$.

By the formula above for $P'$, we have
\begin{equation}\label{H-1}
H_P(z) \ \ge \
|P'(z)| \ = \ |a \prod_{k=1}^t F_k^1(z) F_k^2(z)| \ \ge \ A_1^{S({\mathcal C})} Q_{\mathcal C}
\end{equation}
since for any root $\eta$ of $P'$, $|\xi - \eta| \le \max(|z - \xi|, |z-\eta|) = |z-\eta|$.

To derive other lower bounds for $H(z)$ in terms of $Q_{\mathcal C}$,
we consider the derivatives $P^{(1+\rho_k)}$ of $P'$ where $\rho_k := \sigma_1 + \cdots + \sigma_k$
where $\sigma_k = \sum_{j=1}^{r_k} e_{\xi_j^k}$.
To do this, set 
$$
{\mathcal F}_k(z) \ = \ \prod_{\ell = k+1}^t F_{\ell}(z) 
$$ 
for each $1\le k \le t-1$ and note that
$$
P^{(1+ \rho_k)}(z)/\rho_k! \ = \ a {\mathcal F}_k(z) \ + \ a {\mathcal H}_k(z)
$$
where both ${\mathcal F}_k$ and ${\mathcal H}_k(z)$ are homogeneous functions of degree $d-1 - \rho_k$
($d = {\rm deg}(P)$) in the variables $z-\eta$ as $\eta$ runs over the distinct roots of $P'$.
When $A_1 = 0$, we have $\rho_1 = \sigma_1 = e_{\xi}$ and so 
$P^{(1+e_{\xi})}(z)/e_{\xi}! = a {\mathcal F}_1(z)$ and ${\mathcal H}_1(z) = 0$.

When $A_1>0$,
each term in ${\mathcal H}_k$ has a factor $z-\xi_j^{\ell}$ for some $1\le \ell \le k$ and so 
$|a {\mathcal H}_k(z)| < |a {\mathcal F}_k(z)|$ since $|z-\xi_j^{\ell}| < |z-\eta_{j'}^{\ell'}|$ for any $k+1 \le \ell'$.
Therefore $|P^{(1 + \rho_k)}(z)/\rho_k!| = |a {\mathcal F}_k(z)| = |a {\mathcal F}_k(\xi)|$ since
$|\xi - \xi_j^k| = |\xi - z + z - \xi_j^k| = |z - \xi_j^k|$ for any $k\ge 2$.

Hence 
$$
H_P(z)^{1+\rho_k} \ \ge \ |P^{(1+\rho_k)}(z)/\rho_k! | \ = \ |a {\mathcal F}_k(\xi)| 
$$
and if $f^k := \sigma_k - f_k$, then for all $1\le k \le t-1$,
$$
Q_{\mathcal C} \ \le \ A_k^{f^1 + \cdots f^k} \frac{1}{|\prod_{\ell=k+1}^t F_{\ell}^1(\xi)|} 
|a {\mathcal F}_k(\xi)| \ \le \ A_k^{f^1 +\cdots + f^k} \frac{1}{A_{k+1}^{f_{k+1}+\cdots + f_t}}
H_P(z)^{1+\rho_k}.
$$
The last inequality follows from the previous displayed equation and the fact that $|z - \xi_j^k| = |\xi - \xi_j^k|$
for every $k\ge 2$ as observed before. If $A_1 = 0$, then $f^1 = 0$ and we interpret $A_1^{f^1} = 1$.

Therefore since $A_k < A_{k+1} < A_{k+2}$, we apply the above inequality for $Q_{\mathcal C}$
for $k$ and $k+1$ to conclude that
\begin{equation}\label{Q-k}
Q_{\mathcal C} \ \le \ A_{k+1}^{\rho_k - S({\mathcal C})} H_P(z)^{1+\rho_k} \ \ {\rm and} \ \ 
Q_{\mathcal C} \ \le \ A_{k+1}^{\rho_{k+1} - S({\mathcal C})} H_P(z)^{1+\rho_{k+1}}.
\end{equation}
The first inequality with $k=0$ incorporates \eqref{H-1} if we interpret $\rho_0 = 0$.

We now divide the analysis into cases depending on the size $S({\mathcal C})$ of ${\mathcal C}$.
Suppose $\rho_k < S({\mathcal C}) \le \rho_{k+1}$ for some
$0\le k \le t-1$. Again with the interpretation that $\rho_0 = 0$, we see that any cluster
of roots ${\mathcal C}$ must have a size lying in one of these intervals. With $\rho_k < S({\mathcal C}) \le 
\rho_{k+1}$, we see that the first inequality in \eqref{Q-k} implies
$$
A_{k+1}^{S({\mathcal C})- \rho_k} \, Q_{\mathcal C} \ \le \ H_P(z)^{1+\rho_k} 
$$
and this implies \eqref{Q-aim} when $Q_{\mathcal C}^{-1/(S({\mathcal C}) +1)} \le A_{k+1}$
and therefore we may assume
\begin{equation}\label{assumption-1}
A_{k+1} \ < \ Q_{\mathcal C}^{-1/(S({\mathcal C}) + 1)}.
\end{equation}
When $A_1 = 0$, the reduction to \eqref{assumption-1} when $k=0$ is automatic.

But the second inequality in \eqref{Q-k}, together with \eqref{assumption-1}, implies
$$
Q_{\mathcal C} \ \le \  A_{k+1}^{\rho_{k+1} - S({\mathcal C})} H_P(z)^{1+\rho_{k+1}} \ \le \ 
Q_{\mathcal C}^{-(\rho_{k+1} - S({\mathcal C}))/(S({\mathcal C}) + 1)} \, H_P(z)^{1+\rho_{k+1}}
$$
and this unravels to \eqref{Q-aim}, completing the proof of \eqref{first-step-reduction}.

\subsection*{The second step}

Here we examine the case when $\epsilon = 1$. In this case \eqref{I-single-refined} implies
$|I_P| \le C_d \, p^{-s+1/2}$ when $p>d\ge 2$ where $p^s = H = H_P(z_{*}) = |P'(z_{*})|$ and 
$p^{s-1/2} = J_P(z_{*}) = |P''(z_{*})/2|^{1/2}$
for some $z_{*} \in {\mathbb Z}_p$. 
Furthermore \eqref{P-derivatives} states 
$$
|P^{(k)}(z_{*})/k!| \ \le \ p^{ks - (k-1)} \ \ {\rm for \ every} \ \ 1 \le k \le d. 
$$

Let $\xi$ be a root of $P'$ with the property
$|z_{*} - \xi| \ = \ \min_{\eta} |z_{*} - \eta|$. Our aim is to show
\begin{equation}\label{aim-s}
\max_{\xi \in {\mathcal C}} |a \prod_{\eta\notin {\mathcal C}} (\xi - \eta)^{e_{\eta}} |^{1/(S({\mathcal C})+1)}
 \ \le \ p^{s-1/2} 
\end{equation}
when $p\ge 3$. 
This
will show \eqref{phong-stein} in the case $\epsilon = 1$ since $p>d\ge 2$. We follow the argument (and notation)
in the first step but there is one important, additional ingredient to the argument.

We observe that $A_1 = |z_{*} - \xi| = p^{-s+1}$ when $p\ge 3$. To see this, consider the polynomial
$$
g(z) \ := \ P'(z+ z_{*}) \ = \ P'(z_{*}) + P''(z_{*}) z + [P'''(z_{*})/2] \, z^2 + \cdots 
$$
whose roots are precisely $\{\eta - z_{*}\}$ where $\{\eta\}$ are the roots of $P'$.
The size $|z_{*} - \eta|$ of the roots of $g$ are ordered by $A_1 < A_2 < \cdots < A_t$
where the $A_k$ are defined in first step above with $z$ replaced by $z_{*}$. The 
constant term of $g$ has valuation ${\rm ord}_p(P'(z_{*})) = -s$ and
the linear term has valuation ${\rm ord}_p(P''(z_{*})) = -2s + 1$ since $p\not= 2$.
In general,
${\rm ord}_p(P^{(k)}(z_{*})/k!) \ge -ks + (k-1)$ by \eqref{P-derivatives} and this implies
$$
{\rm ord}_p(P^{(k)}(z_{*})) \ \ge \  {\rm ord}_p(P^{(k)}(z_{*})/k!) \ \ge \ -ks + (k-1).
$$

We consider the Newton polygon (see \cite{K}) of $g$
 in the plane generated by the points 
$$
(i, {\rm ord}_p(b_{i+1}), \ \  i=0, \ldots, d-1 \ \  {\rm where} \ \  g(z) = b_1 + b_2 z + \cdots
+ b_d z^{d-1}.
$$
From the conditions detailed above, ${\rm ord}_p(b_1) = -s, \, {\rm ord}_p(b_2) = -2s+1$ (since $p\not= 2$)
and in general ${\rm ord}_p(b_k) \ge -ks + (k-1)$. Hence we see that the Newton polygon lies above the
line $y = (-s+1)x - s$ {\bf and} the first part of the polygon lies along this line.


A basic result from algebraic number theory (see for example, \cite{K}) relates the slopes
of the polygon for $g$ with the valuations of the roots of $g$. In particular the largest 
(in absolute value) negative slope $-s+1$ gives the valuation 
${\rm ord}_p(z_{*}-\xi) = s-1$ of $z_{*}-\xi$;
that is, $A_1 = |z_{*} - \xi| = p^{-s+1}$ as claimed. 

We fix a subcollection ${\mathcal C}$ of the roots of $P'$ and 
follow the argument from the first step with $z$ replaced by $z_{*}$.
In particular \eqref{aim-s} can be expressed as 
\begin{equation}\label{aim-s-again}
Q_{\mathcal C}^{1/(S({\mathcal C}) +1)} \ \le \ p^{s-1/2}.
\end{equation}
The bound \eqref{H-1} becomes
$$
p^s \ = \ |P'(z_{*})| \ \ge \  |a \prod_{k=1}^t F_k^1(z) F_k^2(z)| \ \ge \ A_1^{S({\mathcal C})} Q_{\mathcal C}
$$
and since $p^{-s+1} = A_1$ when $p\ge 3$,
$$ 
(p^{-s+1})^{S({\mathcal C})} Q_{\mathcal C} \ = \
A_1^{S({\mathcal C})} Q_{\mathcal C} \ \le \ p^s
$$
and this implies \eqref{aim-s-again} since $S({\mathcal C}) \ge 1$.
This completes the second step and the proof of Corollary \ref{cor-phong-stein}.

\section{A structural sublevel set statement}\label{structure}

Here we detail a key sublevel set bound central to the proofs of Theorem \ref{H-main}, 
Theorem \ref{main-bound} and Proposition \ref{congruence-bound}.

\begin{proposition}\label{sublevel-structure-basic}
Let $Q \in {\mathbb Q}_p[X]$ 
and $z \in {\mathbb Z}_p$. Suppose $|Q^{(k)}(z)/k!| \ge 1$ and
$|Q(z)| \le p^{-L}$ for some $L\ge 1$. Then there exists a zero $z_{*}\in {\mathbb Z}_p$
of $Q^{(j)}$ for some $0\le j \le k$ such that $|z - z_{*}| \le p^{-L/k}$.
\end{proposition}

Proposition \ref{sublevel-structure-basic} has the following immediate consequence
which is an extension of Proposition 3.1 in \cite{W-JGA}.

\begin{corollary}\label{sublevel-structure-balls}
Let $Q \in {\mathbb Q}_p[X]$ be a polynomials of degree $d$ and
set 
$$
{\mathcal Z} \ = \ \bigl\{z \in {\mathbb Z}_p: Q^{(j)}(z) = 0 \ {\rm for \ some} \ 0 \le j \le d \bigr\}.
$$ 
Then for $L\ge 1$,
\begin{equation}\label{ball-containment}
\bigl\{z \in {\mathbb Z}_p: |Q(z)| \le p^{-L}, \, |Q^{(k)}(z)/k!| \ge 1 \bigr\} \ \subseteq \ 
\bigcup_{z_{*} \in {\mathcal Z}} B_{p^{-L/k}}(z_{*}).
\end{equation}
Hence there is a constant $C_d$ such that
\begin{equation}\label{sublevel-estimate}
\bigl|\{z \in {\mathbb Z}_p: |Q(z)| \le p^{-L}, \, |Q^{(k)}(z)/k!| \ge 1 \}\bigr| \ \le \ C_d \, p^{-L/k} .
\end{equation}
\end{corollary}

{\bf Remark}: 
The estimate \eqref{sublevel-estimate} {\it scales} in the following sense: suppose
$P\in {\mathbb Q}_p[X]$ is a polynomial of degree $d$. Then for any $n,m \in {\mathbb Z}$,
we have
\begin{equation}\label{sublevel-scale-invariant}
\bigl|\{z \in {\mathbb Z}_p: |P(z)| \le p^{m}, \, |P^{(k)}(z)/k!| \ge p^n \}\bigr| \ \le \ C_d \, p^{(m-n)/k} .
\end{equation}
To see \eqref{sublevel-scale-invariant}, we apply \eqref{sublevel-estimate} to $Q(z) = p^n P(z)$
where $|Q^{(k)}(z)/k!| = p^{-n} |P^{(k)}(z)/k!|$ and $|Q(z)| = p^{-n} |P(z)|$.
Hence in terms of $Q$, the bound \eqref{sublevel-scale-invariant} is
$$
\bigl|\{z \in {\mathbb Z}_p: |Q(z)| \le p^{m-n}, \, |Q^{(k)}(z)/k!| \ge 1 \}\bigr| \ \le \ C_d \, p^{(m-n)/k}
$$
and this follows from \eqref{sublevel-estimate} with $L = n-m$ since we may assume $L\ge 1$ 
(otherwise if $L\le 0$ or $n\le m$, the trivial bound of 1 implies \eqref{sublevel-scale-invariant}).

Proposition \ref{sublevel-structure-basic} is a consequence of the following higher order
Hensel lemma which in turn is a extension of Proposition 2.1 in \cite{W-JGA}.

\begin{lemma}\label{hensel-L} Fix $L\ge 1$. For $\phi(t) = \sum_{j=0}^n c_j t^j \in {\mathbb Q}_p[X]$,
set $\lambda = \max_{0\le j \le n} |c_j|$ and $\lambda_{+} = \max(\lambda, 1)$.
Suppose $t_0 \in {\mathbb Z}_p$ is a point where $\phi^{(k)}(t_0) \not= 0$  for each $1\le k \le L$. 
Set 
$\delta = |\phi(t_0)\phi'(t_0)^{-1}(\phi^{(L)}(t_0)/L!)^{-1}|$ and for $1\le k \le L-1$, set
$$
\delta_k \ = \ |(\phi^{(k+1)}(z_0) \phi^{(k)}(z_0)^{-1})/(k+1) \ \phi(z_0) \phi'(z_0)^{-1}|.
$$
Suppose $\delta_k \le 1,  \, 2\le k \le L-1$ and suppose $\lambda_{+} \delta \le 1$
and $\delta_1  < 1$ when $L\ge 2$. When $L=1$, we suppose $\lambda_{+} \delta < 1$. Then
there is a $t\in {\mathbb Q}_{p}$ such that
$$
(a) \ \phi(t) \ = \ 0 \ \ \ {\rm and} \ \ \ (b) \ |t-t_0| \ \le \  |\phi(t_0)\phi'(t_0)^{-1}|.
$$
\end{lemma}

We postpone the proofs of Proposition \ref{sublevel-structure-basic} and Lemma \ref{hensel-L} to
Section \ref{Prop-structure} and Section \ref{hensel-proof}, respectively.

\section{Completing the proof of Proposition \ref{congruence-bound} -- the upper bound}\label{upper}

As we have seen in Section \ref{lower} establishing the lower bound \eqref{congruence-n=1}, it suffices to prove
\begin{equation}\label{P-congruence}
\sup_{a \in {\mathbb Q}_p} \bigl| \{ {\underline{z}}\in {\mathbb Z}_p^n : \, |P({\underline{z}}) - a| \, \le \, 1 \}\bigr|
\ \le \ C_{d,n} \, H_P^{-1}
\end{equation}
where $H_P = \inf_{{\underline{z}} \in {\mathbb Z}_p^n} H_P({\underline{z}})$ and
$H_P({\underline{z}}) = \max_{|\alpha|\ge 1} (|\partial^{\alpha}P({\underline{z}})/\alpha!|^{1/|\alpha|} )$.

Since $H_P = H_{P-a}$, it suffices to prove $|S| \le C_{d,n} H_P^{-1}$ where 
$$
S \ = \ \{{\underline{z}} \in {\mathbb Z}_p^n : |P({\underline{z}})| \le 1 \} \ \subseteq \ 
\bigcup_{1\le |\alpha|\le d}
\bigl\{ {\underline{z}} \in S  : H({\underline{z}}) = |\partial^{\alpha} P({\underline{z}})/\alpha!|^{1/|\alpha|}
\bigr\}.
$$
Let the sets in the union be denoted by $S_{\alpha}$. 
By \eqref{partial-u}, we have for ${\underline{z}} \in S_{\alpha}$ and $|\alpha|=k$, 
$$
|\partial^{\alpha}P({\underline{z}})/\alpha!| \ \le \ A \, 
\max_{{\underline{u}}\in {\mathcal U}_k} |({\underline{u}}\cdot \nabla)^k P({\underline{z}})/k!|
$$
where ${\mathcal U}_k := \{{\underline{u}}_j : 1\le j \le d(n,k)\}$ are the unit vectors from Lemma
\ref{V-basis}. Also $A = 1$ if $p$ does not
divide any of the coefficients $c_j(\alpha)$ arising in \eqref{partial-u} and if $p$ is a divisor, then
$A = p^m = \max |c_j(\alpha)|$ (the max being taken over all the coefficients
arising in \eqref{partial-u}).
\begin{equation}\label{importantly}
{\rm Importantly, \ we \ have} \ \  A \ \le \ C_{d,n}. 
\end{equation}

Hence for $|\alpha| = k$,
$$
S_{\alpha} \ \subseteq \  \bigcup_{{\underline{u}}\in {\mathcal U}_k} \Big\{
{\underline{z}} \in S : \, A^{-1} H_P^k \le  |({\underline{u}}\cdot \nabla)^k P({\underline{z}})/k!|  \Bigr\}.
$$
Denote the sublevel set on the right by $S_{\alpha, {\underline{u}}}$ so that
\begin{equation}\label{S-alpha}
S \ \subseteq \ \bigcup_{1\le |\alpha| \le d} 
\bigcup_{{\underline{u}}\in {\mathcal U}_{|\alpha|}} {S}_{\alpha,{\underline{u}}} .
\end{equation}

Now fix ${\underline{u}} = (u_1, \ldots, u_n) \in {\mathcal U}_{|\alpha|}$ with $|\alpha| = k$
and consider any matrix $M$, bijectively mapping ${\mathbb Z}_p^n$ onto itself with
$M{\underline{u}} = (1,0,\ldots,0) = {\underline{e}}_1$ 
and $|{\rm det}M| = 1$. Since ${\underline{u}}$ is a unit vector,
then $|u_j| =1$ for some $1\le j\le n$. If $j=1$, we could take for example
$$
M \ = \ u_1^{-1} \begin{pmatrix}
1 & 0 & 0 & \cdots & 0 \\
-u_2 & u_1 & 0 & \cdots & 0 \\
\vdots& & \ddots &  & \vdots \\
-u_n & 0 & 0 & \cdots & u_1
\end{pmatrix} .
$$
We transform ${S}_{\alpha,{\underline{u}}}$ by the change of variables 
${\underline{y}} = M{\underline{z}}$ so that
$$
|{S}_{\alpha,{\underline{u}}}| \ = \ \bigl| \bigl\{ 
{\underline{y}} \in {\mathbb Z}_p^n : \, A^{-1} H_P^k \le  |(\partial/\partial y_1)^k Q({\underline{y}})/k!|,
\ |Q({\underline{y}})| \le 1  \bigr\}\bigr|
$$
where $Q = P\circ M^{-1}$. We will estimate $|{S}_{\alpha,{\underline{u}}}|$ by
fixing ${\underline{y}}' = (y_2, \ldots, y_n)$ and consider the corresponding sublevel set
for the polynomial ${{Q}}_{{\underline{y}}'}(y) := Q(y,{\underline{y}}')$; that is,
$$
|{S}_{\alpha,{\underline{u}}}| \ = \ \int_{{\mathbb Z}_p^{n-1}}
\bigl|\big\{ y \in {\mathbb Z}_p : |{{Q}}_{{\underline{y}}'}(y)| \le 1, \ 
|{Q}_{{\underline{y}}'}^{(k)}(y)/k!| \ge A^{-1} H_P^k \bigr\}\bigr| \, 
dy' .
$$
We employ the scaled bound \eqref{sublevel-scale-invariant} to conclude that
$$
\bigl|\big\{ y \in {\mathbb Z}_p : |{{Q}}_{{\underline{y}}'}(y)| \le 1, \ 
|{Q}_{{\underline{y}}'}^{(k)}(y)/k!| \ge A^{-1} H_P^k \bigr\}\bigr| \ \le \ C_{d,n} \, H_P^{-1}
$$
holds uniformly for ${\underline{y}}' \in {\mathbb Z}_p^{n-1}$. Here we used \eqref{importantly}.
Hence $|S_{\alpha,{\underline{u}}}| \le
C_{d, n} H_P^{-1}$ and so by \eqref{S-alpha}, we have $|S| \le C_{d,n} H_P^{-1}$ which
proves \eqref{P-congruence}, concluding the proof of the upper bound in \eqref{congruence-bound-several}
and the proof of Proposition \ref{congruence-bound}.

\section{The proof of Theorem \ref{main-bound}}\label{1.5-first}

First we make a few reductions. We may assume that $1 \ll_d H$ since otherwise the trivial
bound $|I| \le 1$ implies $|I| \le C_d H^{-1}$ which is stronger than the sought after bound
\eqref{I-several}. Also we may assume that 
$$
P({\underline{z}}) \ = \ \sum_{|\alpha|\ge 1} c_{\alpha} {\underline{z}}^{\alpha}
$$
has no constant term. 

Furthermore we may suppose that $|c_{\alpha}| > 1$ for some $|\alpha|\ge 2$.
In fact, if $|c_{\alpha}| \le 1$ for all $|\alpha|\ge 2$, then $J_P({\underline{z}}) \le 1 < H \le H_P({\underline{z}})$
for all ${\underline{z}} \in {\mathbb Z}_p^n$ and so 
$H_P({\underline{z}}) \equiv |{\underline{c}}|$ where ${\underline{c}} = (c_1, \ldots, c_n)$
are the linear coefficients of $P$. Furthermore,
${\rm e}(P({\underline{z}})) = {\rm e}({\underline{c}} \cdot {\underline{z}})$ and so
$$
I_P(H) \ = \ \int_{{\mathbb Z}_p^n} e({\underline{c}}\cdot{\underline{z}}) \, d{\underline{z}} \ = \ 0
$$
since $1 \ll_d H \le |{\underline{c}}|$ implies $|{\underline{c}}| \ge p$. 

The assumption $|c_{\alpha}| > 1$ for some $|\alpha|\ge 2$ implies $1 < J_P$. In fact let $n$ be the largest
integer such that $|c_{\alpha}| > 1$ for some $|\alpha| = n$ and $|c_{\beta}| \le 1$ for all $|\beta| > n$.
Hence $n\ge 2$ and $|\partial^{\alpha}P({\underline{z}})/\alpha!| \equiv |c_{\alpha}| > 1$, implying
$J_P({\underline{z}}) \ge |c_{\alpha}|^{1/|\alpha|} > 1$ for all ${\underline{z}}\in {\mathbb Z}_p^n$
and so $J_P > 1$. 

Let $r \in {\mathbb Z}$ be defined by $p^{r-1} < J_P \le p^r$. By our reduction to $J_P >1$ above, we have 
$r\ge 1$. 
We decompose 
$$
S_H \ := \ \bigl\{{\underline{z}} \in {\mathbb Z}_p^n : \, H \le H_P({\underline{z}}) \bigr\} \  = \ \bigcup_{s\ge r} E_s
$$ 
where 
$$
E_s \ := \ \bigl\{ {\underline{z}} \in {\mathbb Z}_p^n :  p^{s-1} \ < \ J_P({\underline{z}}) \ \le \ p^s, \ 
H \le H_P({\underline{z}}) \bigr\}.
$$
Also let $t\in {\mathbb Z}$ be defined by $p^{t-1} < H \le p^t$. We have $r\le t$.

We decompose
each $E_s = \cup_{v\ge t} E_{s,v}$ further where
$$
E_{s,v} \ := \ \bigl\{ {\underline{z}} \in E_s :  p^{v-1} \ < \ H_P({\underline{z}}) \ \le \ p^v \, \bigr\}.
$$
If $E_{s,v} \not= \emptyset$, then $p^{s-1} < J_P({\underline{z}}) \le H_P({\underline{z}}) \le p^v$
for some ${\underline{z}}$ which implies $s \le v$. 
We decompose $I_P(H)$ accordingly;
$$
I_P(H) \ = \ \mathop{\sum_{s \ge r, v \ge t}}_{s \le v} \int_{E_{s,v}} {\rm e}(P({\underline{z}})) \, d{\underline{z}} \ =: \ \mathop{\sum_{s \ge r, v\ge t}}_{s\le v} I^{s,v}.
$$

\begin{lemma}\label{Is} For any ${\underline{z}} \in E_{s,v}$, we have 
$J_P({\underline{t}}) = J_P({\underline{z}})$ and $H_P({\underline{t}}) = H_P({\underline{z}})$
whenever $|{\underline{t}} - {\underline{z}}| \le p^{-s}$. 
\end{lemma}

\begin{proof}
For ${\underline{z}}\in E_{s,v}$, we have $p^{s-1} < J_P({\underline{z}}) \le p^s$. 
Let $2\le |\alpha| \le d$ be such that 
$J_P({\underline{z}}) = |\partial^{\alpha}P({\underline{z}})/\alpha!|^{1/|\alpha|}$.

Fix ${\underline{t}} = {\underline{z}} + p^s {\underline{w}}$ with ${\underline{w}} \in {\mathbb Z}_p^n$.
Then for any $\beta$ with $|\beta|\ge 1$, we Taylor expand 
$$
\partial^{\beta}P({\underline{t}})/\beta! \ = \ \partial^{\beta}P({\underline{z}})/\beta! + 
\sum_{|\gamma|\ge 1} 
\bigl[\partial^{\beta + \gamma}P({\underline{z}})/\beta! \gamma!\bigr] p^{|\gamma| s} {\underline{w}}^{\gamma} .
$$
We have $|\beta + \gamma|\ge 2$ for $|\gamma|\ge 1$ and
since $(\beta+\gamma)!/\beta!\gamma! \in {\mathbb N}$, 
$$
\bigl|\partial^{\beta+\gamma} P({\underline{z}})/\beta!\gamma!\bigr| \le 
\bigl|\partial^{\beta+\gamma} P({\underline{z}})/(\beta + \gamma)!\bigr| \le 
J_P({\underline{z}})^{|\beta|+|\gamma|}.
$$
Hence for each term in the above sum,
\begin{equation}\label{j,l}
\bigl|\bigl[\partial^{\beta + \gamma}P({\underline{z}})/\beta! \gamma!\bigr] p^{|\gamma| s} {\underline{w}}^{\gamma}\bigr| \ \le \ (p^{-s} J_P({\underline{z}}) )^{|\gamma|} J_P({\underline{z}})^{|\beta|} \ \le \ 
J_P({\underline{z}})^{|\beta|} \le H_P({\underline{z}})^{|\beta|} .
\end{equation}

If $|\beta| \ge 2$, we have $|\partial^{\beta}P({\underline{z}})/\beta!| \le J_P({\underline{z}})^{|\beta|}$
and so $|\partial^{\beta}P({\underline{t}})/\beta!|^{1/|\beta|} \le J_P({\underline{z}})$ by \eqref{j,l}, 
implying $J_P({\underline{t}}) \le J_P({\underline{z}})$. If $\beta = e_j := (0,\ldots,1,0,\ldots)$ with
$1$ in the $j$th entry so that $\partial^{e_j} = \partial/\partial z_j$, 
then $|\partial^{e_j}P({\underline{z}})| \le H_P({\underline{z}})$
and so $|\partial^{e_j}P({\underline{t}})| \le H_P({\underline{z}})$ by \eqref{j,l}, implying
$H_P({\underline{t}}) \le H_P({\underline{z}})$.

Next suppose that $|\partial^{\gamma}P({\underline{t}})/\gamma!| < J_P+P({\underline{z}})^{|\gamma|}$
for every $|\gamma| \ge 2$. Since ${\underline{z}} = {\underline{t}} - p^s {\underline{w}}$, we
Taylor expand around ${\underline{t}}$ to conclude
$$
\partial^{\alpha}P({\underline{z}})/\alpha! \ = \ \sum_{\beta} 
\bigl[\partial^{\alpha + \beta}P({\underline{t}})/\alpha! \beta!\bigr] (-p^s {\underline{w}})^{\beta}
$$
but now for each term in this sum,
$$
\bigl| \bigl[\partial^{\alpha + \beta}P({\underline{t}})/\alpha! \beta!\bigr] (-p^s {\underline{w}})^{\beta}\bigr| \ <
\ J_P+P({\underline{z}})^{|\alpha|},
$$
implying that 
$J_P({\underline{z}})^{|\alpha|} = |\partial^{\alpha} P({\underline{z}})/\alpha! | < J_P({\underline{z}})^{|\alpha|}$
which is a contradiction. Hence there is a $\gamma$ with $|\gamma| \ge 2$ such that
$$
J_P({\underline{z}}) \le \bigl| \partial^{\gamma}P({\underline{t}})/\gamma!\bigr|^{1/|\gamma|} \ \le \ 
J_P({\underline{t}}),
$$
completing the proof that $J_P({\underline{t}}) = J_P({\underline{z}})$.

If $H_P({\underline{z}}) = J_P({\underline{z}})$, then since 
$J_P({\underline{z}}) = J_P({\underline{t}}) \le H_P({\underline{t}})$, we have
$H_P({\underline{t}}) = H_P({\underline{z}})$ in this case. It remains to consider the case 
$J_P({\underline{z}}) < H_P({\underline{z}})$ and therefore 
$H_P({\underline{z}}) = |\partial^{e_j}P({\underline{z}})|$ for some $j$. As above, we Taylor expand
around ${\underline{t}}$ so that
\begin{equation}\label{taylor-t}
\partial^{e_j}P({\underline{z}}) \ = \ \partial^{e_j} P({\underline{t}}) \ + \sum_{|\beta|\ge 1} 
\bigl[\partial^{e_j + \beta}P({\underline{t}})/\beta!\bigr] (-p^s {\underline{w}})^{\beta} .
\end{equation}
For each term in this sum,
\begin{equation}\label{taylor-t-sum}
\bigl| \bigl[\partial^{e_j + \beta}P({\underline{t}})/\beta!\bigr] (-p^s {\underline{w}})^{\beta}\bigr| \ \le
\ J_P({\underline{t}})^{|\beta| +1} p^{-|\beta| s} \ \le \ J_P({\underline{z}}) \ < \ H_P({\underline{z}})
\end{equation}
since $J_P({\underline{t}}) = J_P({\underline{z}})$. Hence if 
$|\partial^{e_j}P({\underline{t}})| < H_P({\underline{z}})$, then \eqref{taylor-t} and \eqref{taylor-t-sum} implies
$H_P({\underline{z}}) = |\partial^{e_j}P({\underline{z}})| < H_P({\underline{z}})$ which is a contradiction.
Therefore $H_P({\underline{z}}) \le |\partial^{e_j}P({\underline{t}})| \le H_P({\underline{t}})$, completing
the proof that $H_P({\underline{t}}) = H_P({\underline{z}})$.
\end{proof} 
 
For every $s\ge r, v \ge t$ with $s\le v$, we set $R_s = [{\mathbb Z}/p^s{\mathbb Z}]^n$ and 
$$
{\mathcal E}_{s,v} \ = \ \bigl\{{\underline{t}}\in R_s: p^{s-1} < J_P({\underline{t}}) \le p^s, \ p^{v-1} <
H_P({\underline{t}}) \le p^v \,  \bigr\}
$$  
with the added condition $H \le H_P({\underline{t}})$ when $v=t$.
By Lemma \ref{Is}, we see that 
$$
E_{s,v} \ = \ \bigcup_{{\underline{t}}\in {\mathcal E}_{s,v}} B_{p^{-s}}({\underline{t}})
$$
and so
$$
I^{s,v} \ = \ \sum_{{\underline{t}} \in {\mathcal E}_{s,v}} \, \int_{B_{p^{-s}}({\underline{t}})} 
{\rm e}(P({\underline{z}})) \, d{\underline{z}} \ =: \ 
\sum_{{\underline{t}}\in {\mathcal E}_{s,v}} I^{s,v}_{\underline{t}}.
$$


\begin{lemma}\label{s<v}
For each ${\underline{t}} \in {\mathcal E}_{s,v}$ with $s<v$, \ $I^{s,v}_{\underline{t}} = 0$.
\end{lemma}

\begin{proof}
Fix ${\underline{t}} \in {\mathcal E}_{s,v}$ with $s<v$. Then
$J_P({\underline{t}}) \le p^s \le p^{v-1} < H_P({\underline{t}})$ which implies 
$H_P({\underline{t}}) = |\nabla P({\underline{t}})|$ and so $|\nabla P({\underline{t}})| = p^v$. 

For ${\underline{z}} = {\underline{t}} + p^s {\underline{w}}$,
Taylor expand
$$
P({\underline{z}}) = P({\underline{t}}) + \nabla P({\underline{t}}) \cdot p^s {\underline{w}} \ + \ E
$$
where
$$
E \ = \ \sum_{|\alpha|\ge 2} \bigl[\partial^{\alpha}P({\underline{t}})/\alpha!\bigr] (p^s {\underline{w}})^{\alpha} \ \in \ {\mathbb Z}_p
$$
since $J_P({\underline{t}}) \le p^s$. Hence
$$
I^{s,v}_{\underline{t}} \ = \ p^{-sn} {\rm e}(P({\underline{t}})) \, \int_{{\mathbb Z}_p^n} 
{\rm e}(\nabla P({\underline{t}}) \cdot {\underline{w}}) \, d{\underline{w}} \ = \ 0
$$
since $p^{-v} | \, \nabla P({\underline{t}})$ and $v \ge 1$. 
\end{proof}

Therefore 
$$
I_P(H) \ = \ \sum_{s\ge t} \, \sum_{{\underline{t}}\in {\mathcal E}_s'} \int_{B_{p^{-s}}({\underline{t}})}
{\rm e}(P({\underline{z}})) \, d{\underline{z}} \ =: \ \sum_{s\ge t} I^s_{{\underline{t}}}
$$
where ${\mathcal E}_s' = \{ {\underline{z}} :  p^{-s+1} < J_P({\underline{t}}) \le H_P({\underline{t}}) \le p^s \}$
with the added condition $H \le H_P({\underline{t}})$ when $s = t$.  

We set
${\mathcal E} := \{s : {\mathcal E}_s' \not= \emptyset\}$. Then $r\le t \le s_{*} := \inf {\mathcal E}$. 
Note that if ${\underline{t}} \in {\mathcal E}_s'$, then $|\nabla P({\underline{z}})| \le p^s$
for any $|{\underline{z}} - {\underline{t}}| \le p^{-s}$. In fact, Taylor expand
\begin{equation}\label{grad-constant}
(\partial/\partial z_j) P({\underline{z}}) \ = \ (\partial/\partial z_j) P({\underline{t}}) \ + \
\sum_{|\alpha|\ge 1} 
\bigl[\partial^{\alpha +e_j} P ({\underline{t}})/(\alpha + e_j)!\bigr] (p^s {\underline{w}})^{\alpha},
\end{equation}
noting each term in the above sum has absolute value at most 
$(J({\underline{t}}) p^{-s})^{|\alpha|} J({\underline{t}}) \le p^s$. Therefore 
$|\nabla P({\underline{z}})|\le p^s$.

Hence
\begin{equation}\label{I-s}
|I^s| \ := \ \bigl| \sum_{{\underline{t}} \in {\mathcal E}_s'} I^s_{\underline{t}} \bigr| \ \le \ 
\bigl|\{{\underline{z}}\in {\mathbb Z}_p^n : p^{s-1} < J_P({\underline{z}}) \le p^s, \ 
|\nabla P({\underline{z}})| \le p^s \} \bigr|.
\end{equation}
Let ${\mathcal S}_s$ denote the sublevel set on the right above. 

If ${\underline{z}} \in {\mathcal S}_s$,
then there is an $2\le |\alpha| = k$ such that 
$|\partial^{\alpha}P({\underline{z}})/\alpha!| =: p^{ks - j} > p^{k(s-1)}$ (so $0\le j \le k-1$).
By \eqref{partial-u}, we have
$$
p^{k(s-1)} \ < \ |\partial^{\alpha}P({\underline{z}})/\alpha!| \ \le \ A \, 
\max_{{\underline{u}}\in {\mathcal U}_k} |({\underline{u}}\cdot \nabla)^k P({\underline{z}})/k!|
$$
where we recall ${\mathcal U}_k := \{{\underline{u}}_j : 1\le j \le d(n,k)\}$ are the unit vectors from Lemma
\ref{V-basis}.

Finally note that $|{\underline{u}}\cdot \nabla)P({\underline{z}})| \le |\nabla P({\underline{z}})| \le p^s$
and so
$$
{\mathcal S}_s \ \subseteq \ \bigcup_{k=2}^d \bigcup_{{\underline{u}}\in {\mathcal U}_k} \Big\{
{\underline{z}} \in {\mathbb Z}_p^n : \, p^{ks-(k-1)} \le A |({\underline{u}}\cdot \nabla)^k P({\underline{z}})/k!|,
\ |({\underline{u}}\cdot \nabla)P({\underline{z}})| \le p^s   \Bigr\}.
$$
Denote the sublevel set on the right by ${\mathcal S}_{s, {\underline{u}}}$ so that
\begin{equation}\label{S-s}
{\mathcal S}_s \ \subseteq \ \bigcup_{k=2}^d 
\bigcup_{{\underline{u}}\in {\mathcal U}_k} {\mathcal S}_{s,{\underline{u}}} .
\end{equation}

As before we fix ${\underline{u}} = (u_1, \ldots, u_n) \in {\mathcal U} := \cup_{k\ge 2} \, {\mathcal U}_k$
and consider any matrix $M$, bijectively mapping ${\mathbb Z}_p^n$ onto itself with
$M{\underline{u}} = (1,0,\ldots,0) = {\underline{e}}_1$ 
and $|{\rm det}M| = 1$. 
We transform ${\mathcal S}_{s,{\underline{u}}}$ by the change of variables 
${\underline{y}} = M{\underline{z}}$ so that
$$
|{\mathcal S}_{s,{\underline{u}}}| \ = \ \bigl| \bigl\{ 
{\underline{y}} \in {\mathbb Z}_p^n : \, p^{ks-(k-1)} \le A |(\partial/\partial y_1)^k Q({\underline{y}})/k!|,
\ |(\partial/\partial y_1)Q({\underline{y}})| \le p^s   \bigr\}\bigr|
$$
where $Q = P\circ M^{-1}$. We will estimate $|{\mathcal S}_{s,{\underline{u}}}|$ by
fixing ${\underline{y}}' = (y_2, \ldots, y_n)$ and consider the corresponding sublevel set
for the polynomial ${\tilde{Q}}(y) := (\partial/\partial y) Q(y,{\underline{y}}')$; that is, we will
estimate the measure of
\begin{equation}\label{Q-tilde}
\big\{ y \in {\mathbb Z}_p : |{\tilde{Q}}(y)| \le p^s, \ |{\tilde{Q}}^{(k-1)}(y)/k!| \ge A^{-1} p^{ks-(k-1)} \bigr\}
\end{equation}
by employing the scale-invariant bound \eqref{sublevel-scale-invariant}. 

In fact by applying the bound \eqref{sublevel-scale-invariant} to the sublevel set in \eqref{Q-tilde}, we have
$$
\bigl| \big\{ y \in {\mathbb Z}_p : |{\tilde{Q}}(y)| \le p^s, \ |{\tilde{Q}}^{(k-1)}(y)/k!| \ge A^{-1} p^{ks-(k-1)} \bigr\} \bigr| \ \le \ C_d [|k|^{-1}A]^{1/(k-1)} p^{-s+1}
$$
since
$$
\Bigl[ \frac{p^s}{|k|A^{-1}p^{ks - (k-1)}}\Bigr]^{1/(k-1)} \ = \ [|k|^{-1}A]^{1/(k-1)} \, p^{-s+1}.
$$
Integrating out ${\underline{y}}' = (y_2, \ldots, y_n)$ over ${\mathbb Z}_p^{n-1}$ gives
$$
|{\mathcal S}_{s, {\underline{u}}}| \ \le \ C_d \, [|k|^{-1}A]^{1/(k-1)} \, p^{-s+1}
$$
and so by \eqref{I-s} and \eqref{S-s} (and the fact that $|k|^{-1} A \lesssim_{d,n} 1$ by \eqref{importantly}),
$$
|I^s| \ \le \ |{\mathcal S}_s| \ \le \ C_{d,n} \, p^{-s+1},
$$
implying
\begin{equation}\label{I-several-bound}
|I_P(H)| \ = \ |\sum_{s \in {\mathcal E}} I^s | \ \le \ C_{d,n} \, p \ 
\sum_{s\in {\mathcal E}} \, p^{-s} \ 
\le \ C_{d,n}  \ p^{-s_{*} + 1} \ \le \ C_{d,n} \, p \ H^{-1}
\end{equation}
since $H \le p^t \le p^{s_{*}}$, completing the proof of Theorem \ref{main-bound}.

\section{The proof of Theorem \ref{H-main} - the first part}\label{H-first}

We now restrict our attention to $n=1$ and polynomials $P(z)$ of a single variable $z$.

If $s_{*} \ge t+1$, then \eqref{I-several-bound} implies $|I_P(H)| \le C_{d} \, p^{-t} \le C_{d}\, H^{-1}$ which implies the desired bound \eqref{I-single}.
Therefore we may assume $s_{*} =\inf {\mathcal E} = t$ and so
\begin{equation}\label{I-reduction}
I_P(H) \ = \ I^{s_{*}} \ + \ O_{d}(H^{-1}) \ = \ \sum_{{\underline{t}}\in {\mathcal E}_{s_{*}}'} 
I^{s_{*}}_{\underline{t}} \ + \ O_{d}(H^{-1}) .
\end{equation}

Since $n=1$, our variables ${\underline{z}} = z \in {\mathbb Z}_p$ and ${\underline{t}} = t \in R_s$
only have a single component. Recall that
$$
{\mathcal E}_{s_{*}}' \ = \ \{t \in R_{s_{*}} : p^{s_{*}-1} < J_P(t) \le p^{s_{*}}, \, |P'(t)| \le p^{s_{*}} \
{\rm and} \ H \le H_P(t) \}.
$$

We set $\lambda = \max |c_j|$ where
$P(z) = c_d z^d + \cdots + c_1 z$. Our reduction to $1 \ll_d H$ implies $H\le \lambda$ and so
in this case, $p\le \lambda$. Let $n$ be the largest integer $n$ such that $|c_n| = \lambda$
and $|c_j| < \lambda$ for all $j>n$. Then $|P^{(n)}(z)/n!| \equiv |c_n| = \lambda$ and furthermore
$|P^{(j)}(z)/j!| < \lambda^{1/j} < \lambda^{1/n} = |P^{(n)}(z)/n!|^{1/n}$ for each $j>n$, implying
that $J_P(z) = J'(z) := \max_{2\le k \le n} |P^{(k)}(z)/k!|^{1/k}$.

By Lemma \ref{Is} and \eqref{grad-constant}, we have
$$
I^{s_{*}} \ = \ \int_{E_{s_{*}}'} {\rm e}(P(z)) \, dz
$$
where (recall $J_P(z) = J'(z)$)
$$
E_{s_{*}}' \ = \ \bigl\{ z \in {\mathbb Z}_p : p^{s_{*}-1} < J_P(z) \le p^{s_{*}}, \
|P'(z)| \le p^{s_{*}} \ {\rm and} \ H \le H_P(z) \bigr\}.
$$ 

We decompose $E_{s_{*}}' = F_0 \cup \cdots \cup F_{n-2}$ where\footnote{In this section,
we are suppressing dividing $P^{(\ell)}(z)$ by $\ell!$ 
in the terms $|P^{(\ell)}(z)/\ell!|^{1/\ell}$ defining $J_P(z)$ for notational convenience although keeping
the factorials is more natural.}  
$$
F_0 := \{ z \in E_{s_{*}}' : J_P(z) = |P^{(n)}(z)|^{1/n} \}, \  F_1 :=  
\bigl\{ z\in E_{s_{*}}' \setminus F_0 :  J_P(z) = |P^{(n-1)}(z)|^{1/(n-1)} \bigr\} 
$$
and inductively,
$$
 F_k \ := \ \bigl\{ z \in E_{s_{*}}' \setminus \cup_{j=0}^{k-1} F_j : J_P(z) \ = \  |P^{(n-k)}(z)|^{1/(n-k)} \bigr\} .
$$

For $z \in F_0$, we have
$$
|P^{(n-k)}(z)|^{\frac{1}{n-k}} \ \le \ |P^{(n)}(z)|^{\frac{1}{n}}
$$
for every $0\le k \le n-2$. We apply Lemma \ref{hensel-L} with $L=1$
to $\phi(t) = P^{(n-1)}(t)$ and $t_0 = z$. We have $\delta = |\phi(t_0) \phi'(t_0)^{-2}|$
and so
$$
\delta \le  |P^{(n)}(z)|^{(n-1)/n} |P^{(n)}(z)|^{-2} \ = \ |P^{(n)}(z)|^{-(n+1)/n} .
$$
Since $\lambda = |P^{(n)}(z)|$, we see that 
$$
\lambda \delta \ \le \  |P^{(n)}(z)|^{-1/n} \ = \lambda^{-1/n} \ < \ 1
$$ 
and so
we conclude there exists a zero $z_{*} \in {\mathbb Q}_p$ of $P^{(n-1)}$ with 
$$
|z_{*} - z| \ \le \  |\phi(t_0) \phi'(t_0)^{-1}| \ \le \ |P^{(n)}(z)|^{-1/n}  \ = \ J_P(z)^{-1} \ < p^{-s_{*}+1} .
$$ 
Therefore $|z_{*} - z| \le p^{-s_{*}}$ implying $z_{*} \in {\mathbb Z}_p$ and 
\begin{equation}\label{F0}
F_0 \ \subseteq \ \bigcup_{w\in {\mathcal Z}} B_{p^{-s_{*}}}(w)
\end{equation}
where ${\mathcal Z} = \{ w \in {\mathbb Z}_p : P^{(j)}(w) = 0 \ {\rm for \ some} \ 1\le j \le d\}$.

Next we consider $z\in F_1$ so that
$$
|P^{(n-k)}(z)|^{\frac{1}{n-k}} \ \le \  |P^{(n-1)}(z)|^{\frac{1}{n-1}}
$$
for every $0\le k \le n-2$ with strict inequality when $k=0$. We apply Lemma \ref{hensel-L} with $L=2$ to
$\phi(t) = P^{(n-2)}(t)$ and $t_0 = z$. Recall that in this case, 
$\delta = \phi(t_0) \phi'(t_0)^{-1} \phi''(t_0)^{-1}$ and so
$$
\delta \ \le \ |P^{(n-1)}(z)|^{-1/(n-1)} |P^{(n)}(z)|^{-1},
$$
implying 
$$
\lambda \delta \ \le \ |P^{(n-1)}(z)|^{-1/(n-1)} \ = \ J_P(z)^{-1}  \ < \ p^{-s_{*}+1}
$$
and so $\lambda \delta \le p^{-s_{*}} < 1$ since $1\le r \le s_{*}$.

We also need to verify the condition $\delta_1 < 1$ where 
$\delta_1 = \phi(t_0) \phi'(t_0)^{-2} \phi''(t_0)$.
We have
$$
\delta_1 \ \le \ |P^{(n-1)}(z)|^{-n/(n-1)} |P^{(n)}(z)| \ < \ 1
$$
and so $\delta_1 < 1$.
Hence there exists a zero $z_{*}$ of $P^{(n-2)}$ with 
$$
|z_{*} - z| \ \le \ |\phi(t_0) \phi'(t_0)^{-1}| \ \le \  |P^{(n-1)}(z)|^{-1/(n-1)} \ = \ J_P(z)^{-1} < p^{-s_{*}+1} 
$$ 
and so $|z_{*} - z| \le p^{-s_{*}}$ implying $z_{*} \in {\mathbb Z}_p$. Hence
\begin{equation}\label{F1}
F_1 \ \subseteq \ \bigcup_{w\in {\mathcal Z}} B_{p^{-s_{*}}}(w) .
\end{equation}

For general $z\in F_k$ with $k < n-2$ (the case $k=n-2$ is special), we have
$$
|P^{(n-r)}(z)|^{\frac{1}{n-r}} \ \le \ |P^{(n-k)}(z)|^{\frac{1}{n-k}}
$$
for every $0\le r \le n-2$ with strict inequality for $0\le r \le k-1$. It is natural to apply Lemma \ref{hensel-L} with $L = k+1$ to
$\phi(t) = P^{(n-k-1)}(t)$ and $t_0 = z$. First let us verify the condition $\lambda \delta < 1$
where 
$$
\delta \ = \ |\phi(t_0) \phi'(t_0)^{-1} \phi^{(k+1)}(t_0)^{-1}| \ = \ |P^{(n-k-1)}(z)|  |P^{(n-k)}(z)
P^{(n)}(z)|^{-1} .
$$
We have
$$
\lambda \delta \ \le \ \lambda |P^{(n-k)}(z)|^{-1/(n-k)} |P^{(n)}(z)|^{-1} \ = \
|P^{(n-k)}(z)|^{-1/(n-k)} \ = \ J_P(z)^{-1}
$$
implying $\lambda \delta < p^{-s_{*}+1}$ and so $\lambda \delta \le p^{-s_{*}} < 1$. 

Next let us verify that $\delta_{1} < 1$.
$$
\delta_1 \ = \ |\phi(t_0) \phi'(t_0)^{-2} \phi''(t_0)| \ = \
 |P^{(n-k-1)}(z)|  |P^{(n-k)}(z)|^{-2} |P^{(n-k+1)}(z)| .
$$
Since $|P^{(n-k+1)}(z)|^{1/(n-k+1)} < |P^{(n-k)}(z)|^{1/(n-k)}$, we have
$$
\delta_1 \ \le \   |P^{(n-k)}(z)|^{-(n-k+1)/(n-k)} |P^{(n-k+1)}(z)| \ < \ 1.
$$

Unfortunately it is {\bf not} the case that for $2 \le \ell \le L-1$, $\delta_{\ell} \le 1$
where
$$
\delta_{\ell}\  := |\phi(t_0) \phi'(t_0)^{-1} \phi^{(\ell+1)}(t_0) \phi^{(\ell)}(t_0)^{-1}|.
$$
Note that in our case,
$$
\delta_{\ell} \ = \ |P^{(n-k-1)}(z) P^{(n-k)}(z)^{-1}  P^{(n-k+\ell)}(z) P^{(n-k+\ell - 1)}(z)^{-1}|
$$
and so
$$
\delta_{\ell} \ \le \ |P^{(n-k)}(z)|^{-1/(n-k)} | P^{(n-k+\ell)}(z) P^{(n-k+\ell - 1)}(z)^{-1}|.
$$
We would need
$$
 | P^{(n-k+\ell)}(z) P^{(n-k+\ell - 1)}(z)^{-1}| \ \le \  |P^{(n-k)}(z)|^{1/(n-k)} \eqno{(S_{\ell})}
$$
to hold for every $2\le \ell \le k$. But $(S_{\ell})$ does not necessarily hold for any $\ell$. 

Therefore we will decompose the set $F_k$ further when $k\ge 2$. We write 
$F_k = F_k^1 \cup \cdots \cup F_k^{k}$
and define $F_k^r$ inductively. First
$$
F_k^1 \ := \ \bigl\{ z \in F_k : |P^{(n- 1)}(z)| |P^{(n-k)}(z)|^{1/(n-k)} <  | P^{(n)}(z) | \bigr\}.
$$
Hence if $z \notin F_k^1$, then $(S_k)$ holds.

Next we define
$$
F_k^2 \ := \ \bigl\{ z \in F_k \setminus F_k^1 : |P^{(n- 2)}(z)| |P^{(n-k)}(z)|^{1/(n-k)} <  | P^{(n-1)}(z) | \bigr\}.
$$
Hence if $z \notin [F_k^1 \cup F_k^2]$, then both
$(S_k)$ and $(S_{k-1})$ hold.

In general, for $1\le r \le k-1$, we define
$$
F_k^r \ := \ \bigl\{ z \in F_k \setminus [F_k^1 \cup \cdots \cup F_k^{r-1}] : 
|P^{(n- r)}(z)| |P^{(n-k)}(z)|^{1/(n-k)} <  | P^{(n-r+1)}(x) | \bigr\}.
$$
Hence if $z \notin [F_k^1 \cup \cdots \cup F_k^r]$, then 
$(S_{\ell})$ holds for every $k-r+1\le \ell \le k$.

Finally, we define 
$$
F_k^{k} \ := \ \{ z \in F_k \setminus [F_k^1 \cup \cdots \cup F_k^{k-1} \}
$$
so that for $z \in F_k^k$, the property $(S_{\ell})$ holds for every $2\le \ell \le k$. Hence for $z \in F_k^k$,
we can successfully apply Lemma \ref{hensel-L} to find a zero $z_{*} \in {\mathbb Q}_p$ of $P^{(n-k-1)}$
such that 
$$
|z_{*} - z| \ \le \  |\phi(t_0) \phi'(t_0)^{-1}| \ \le \  |P^{(n-k)}(z)|^{-1/(n-k)} 
\  = \ J(z)^{-1} \ < \ p^{-s_{*} +1}
$$ 
and so $|z_{*} - z| \le p^{-s_{*}}$ which implies $z_{*} \in {\mathbb Z}_p$. Hence
\begin{equation}\label{Fkk}
F_k^k \ \subseteq \ \bigcup_{w\in {\mathcal Z}} B_{p^{-s_{*}}}(w) .
\end{equation}

We now examine $z \in F_k^1$. Here we apply Lemma \ref{hensel-L} with $L=1$ to
$\phi(t) = P^{(n-1)}(t)$ and $t_0 = z$. In this case $\delta = \phi(t_0) \phi'(t_0)^{-2}$
so that
$$
\delta \ = \ |P^{(n-1)}(z) P^{(n)}(z)^{-2}| \ < \ |P^{(n-k)}(z)|^{-1/(n-k)} |P^{(n)}(z)|^{-1} 
$$
and so
$$
\lambda \delta \ < \  |P^{(n-k)}(z)|^{-1/(n-k)} \ = \ J_P(z)^{-1} \ \le p^{-s_{*}+1},
$$
implying $\lambda \delta \le p^{-s_{*}} < 1$.
Hence there exists a zero $z_{*} \in {\mathbb Q}_p$ of $P^{(n-1)}$ such that
$$
|z_{*} - z| \ \le \ |\phi(t_0) \phi'(t_0)^{-1}| \ < \ |P^{(n-k)}(z)|^{-1/(n-k)} \ = \  J_P(z)^{-1} \ \le \ p^{-s_{*}+1}
$$
and so $|z_{*} - z| \le p^{-s_{*}}$ which implies $z_{*} \in {\mathbb Z}_p$. Hence
\begin{equation}\label{Fk1}
F_k^1 \ \subseteq \ \bigcup_{w\in {\mathcal Z}} B_{p^{-s_{*}}}(w) .
\end{equation}

For $z \in F_k^r, 2\le r \le k-1$, we have $z \notin [F_k^1 \cup \cdots \cup F_k^{r-1}]$ and so
$(S_{\ell})$ holds for $k-r +2 \le \ell \le k$.
We will apply Lemma \ref{hensel-L} with $L=r$ to
$\phi(t) = P^{(n-r)}(t)$ and $t_0 = z$. In this case, 
$$
\delta \ = \ |\phi(t_0)\phi'(t_0)^{-1} \phi^{(L)}(t_0)^{-1}| 
$$
and so
$$
\delta \ = \
|P^{(n-r)}(z) P^{(n-r+1)}(z)^{-1} P^{(n)}(z)^{-1}| \ < \  |P^{(n-k)}(z)|^{-1/(n-k)} |P^{(n)}(z)^{-1}|.
$$
Hence
$$
\lambda \delta \ < \  |P^{(n-k)}(z)|^{-1/(n-k)} \ = \  J_P(z)^{-1} \ < \ p^{-s_{*}+1} 
$$
and so $\lambda \delta \le p^{-s_{*}} < 1$.

Next we need to control, for $1\le \ell \le r-1$,
$$
\delta_{\ell} \ = \ |\phi(t_0) \phi'(t_0)^{-1} \phi^{(\ell+1)}(t_0) \phi^{(\ell)}(t_0)^{-1}|
$$
or 
$$
\delta_{\ell} \ = \ |P^{(n-r)}(t_j) P^{(n-r+1)}(t_j)^{-1} P^{(n-r+\ell + 1)}(t_j) P^{(n-r+\ell)}(t_j)^{-1}|.
$$
Note that
$$
\delta_{\ell} \ < \  |P^{(n-k)}(z)|^{-1/(n-k)} |P^{(n-r+\ell + 1)}(z) P^{(n-r+\ell)}(z)^{-1}|
$$
and so, when $\ell = 1$, using $(S_{k-r+2})$,
$$
\delta_1 \ < \  |P^{(n-k)}(z)|^{-1/(n-k)} |P^{(n-k)}(z)|^{1/(n-k)} \ = \ 1.
$$
For $2\le \ell \le r-1$, using $(S_{k-r+\ell +1})$,
$$
\delta_{\ell} \ < \ |P^{(n-k)}(z)|^{-1/(n-k)} |P^{(n-k)}(z)|^{1/(n-k)} \ = \ 1.
$$
Therefore Lemma \ref{hensel-L} gives us
a zero $z_{*} \in {\mathbb Q}_p$ of $P^{(n-r)}$ such that
$$
|z_{*} - z| \ \le \  |\phi(t_0) \phi'(t_0)^{-1}| \ \le \ |P^{(n-k)}(z)|^{-1/(n-k)} \ = \ J_P(z)^{-1} \ < \
p^{-s_{*}+1}
$$
and so $|z_{*} - z| \le p^{-s_{*}}$. Hence $z_{*} \in {\mathbb Z}_p$ and
\begin{equation}\label{Fkr}
F_k^r \ \subseteq \ \bigcup_{w\in {\mathcal Z}} B_{p^{-s_{*}}}(w) .
\end{equation}

From \eqref{F0}, \eqref{F1}, \eqref{Fkk}, \eqref{Fk1} and \eqref{Fkr}, we see that
$$
\bigl| F_0 \cup F_1 \cup \cdots \cup F_{n-3} \bigr | \ \le \ C_d \, p^{-s_{*}} \ \le \ C_d \, H^{-1}
$$
and so
\begin{equation}\label{I-F(n-2)}
I^{s_{*}} \ = \ \int_{F_{n-2}} {\rm e} (P(z)) \, dz \ + \ O_d(H^{-1}).
\end{equation}

\section{Proof of Theorem \ref{H-main} - the final part}\label{H-final}

For $z \in F_{n-2}$, we have
\begin{equation}\label{F(n-2)-condition}
|P^{(n)}(z)/n!|^{1/n} 
< \cdots < |P'''(z)/6|^{1/3} < |P''(z)/2|^{1/2} \ = \ J_P(z) \ \le \ p^{s_{*}},
\end{equation}
with $|P'(z)| \le p^{s_{*}}$, \ $p^{s_{*}-1} < J_P(z)$ and $H\le H_P(z)$. 
It is not necessarily true that $|P'(z)| \le |P''(z)/2|^{1/2}$ but if this is the case, then the argument in the previous
section works.

So we decompose
$F_{n-2} = F_{n-2}' \cup F_{n-2}''$ where $F_{n-2}' = \{z\in F_{n-2}: |P'(z)| \le |P''(z)/2|^{1/2}\}$
and hence for $z \in F_{n-2}'$, the previous argument applies and gives us a zero
$z_{*} \in {\mathbb Q}_p$ with $|z_{*} - z| \le J_P(z)^{-1} < p^{-s_{*}+1}$, implying
$|z_{*} - z| \le p^{-s_{*}}$. Hence $z_{*}\in {\mathbb Z}_p$ and
$$
F_{n-2}' \subseteq \bigcup_{w\in {\mathcal Z}} B_{p^{-s_{*}}}(w) \ \ {\rm which \ implies} \
|F_{n-2}'| \ \le \ C_d \, p^{-s_{*}} \ \le \ C_d \, H^{-1}
$$
and so \eqref{I-F(n-2)} refines to
\begin{equation}\label{I-F(n-2)''}
I^{s_{*}} \ = \ \int_{F_{n-2}''} {\rm e} (P(z)) \, dz \ + \ O_d(H^{-1}).
\end{equation}

For $z\in F_{n-2}''$, we have $|P''(z)/2|^{1/2} < |P'(z)| \le p^{s_{*}}$, implying
$|P''(z)/2| = p^{2s_{*}-1}$ and $|P'(z)| = p^{s_{*}}$. Hence $H_P(z) \equiv p^{s_{*}}$
for $z \in F_{n-2}''$ and we can drop the condition $H \le H_P(z)$ for $z \in F_{n-2}''$. Furthermore
\eqref{F(n-2)-condition} implies
\begin{equation}\label{F(n-2)-conditionv2}
|P^{(k)}(z)/k!| \ \le \ p^{k s^{*} - (k-1)} \ {\rm for} \ 1\le k \le n
\end{equation}
with equality for $k=1$ and $k=2$. The condition \eqref{F(n-2)-condition} also implies
\begin{equation}\label{F(n-2)-conditionv3}
|P^{(k)}(z)/k!| \ < \ p^{s_{*}} \, |P^{(k-1)}(z)/(k-1)!| \ \ {\rm for} \ 2 \le k \le n
\end{equation}
and $z \in F_{n-2}''$. This follows by the following simple induction agument. 

Suppose \eqref{F(n-2)-conditionv2}
and \eqref{F(n-2)-conditionv3} hold for $1 \le j \le k -1$ and set $p^m = |P^{(k}(z)/k!|$
and $p^{\ell} = |P^{(k-1)}(z)/(k-1)!|$ so that $\ell \le (k-1)s_{*} - (k-2)$ holds. 
The condition \eqref{F(n-2)-condition} implies $m/k < \ell/(k-1)$ or 
$$
m < \ell + \ell/(k-1) \le \ell + s_{*} - (k-2)/(k-1) \ < \ \ell + s_{*}
$$
which says $|P^{(k)}(z)/k!| < p^{s_{*}} |P^{(k-1)}(z)/(k-1)!|$. Also
$$
m < \ell + \ell/(k-1) \ \le \ (k-1)s_{*} - (k-2) + s_{*} - (k-2)/(k-1) \ = \ k s_{*} -k + 2 - (k-2)/(k-1),
$$
implying $m \le k s_{*} - k + 1 = k s_{*} - (k-1)$. This establishes
\eqref{F(n-2)-conditionv2} and \eqref{F(n-2)-conditionv3}.

Suppose that $z_0 \in F_{n-2}''$ and consider $z = z_0 + p^{s_{*}}w$ for $w\in {\mathbb Z}_p$.
Arguing as in Lemma \ref{Is}, we Taylor expand (for any $1\le k \le n$)
$$
P^{(k)}(z)/k! \ = \ P^{(k)}(z_0)/k! + \sum_{\ell=1}^{n-k} [P^{(k+\ell)}(z_0)/k!\ell!] (p^{s_{*}} w)^{\ell} \ + \ E
$$
where 
$$
|E| \ < \ \lambda \, p^{-(n-k+1)s_{*}} \ = \ |P^{(n)}(z_0)/n!| \ p^{-(n-k+1)s_{*}} 
$$
and by \eqref{F(n-2)-conditionv3},
$|P^{(n)}(z_0)/n!| < p^{(n-k)s_{*}} |P^{(k)}(z_0)/k!|$ so that
$$
|E| \ < \ p^{(n-k)s_{*}} \, |P^{(k)}(z_0)/k!| \,
p^{-(n-k)s_{*}} p^{-s_{*}} \ = \ p^{-s_{*}} \ |P^{(k)}(z_0)/k!|.
$$
Also each term in the sum satisfies 
$$
|[P^{(k+\ell)}(z_0)/k!\ell!] (p^{s_{*}} w)^{\ell}| \ < \  p^{\ell s_{*}} |P^{(k)}(z_0)/k!| p^{-\ell s_{*}}
\ = \ |P^{(k)}(z_0)/k!|
$$
by \eqref{F(n-2)-conditionv3}.
Hence $|P^{(k)}(z)/k!| = |P^{(k)}(z_0)/k!|$.

Therefore from \eqref{I-F(n-2)''},
\begin{equation}\label{I-calF}
I^{s_{*}} \ = \ \sum_{t\in {\mathcal F}} \ \int_{B_{p^{-s_{*}}}(t)} {\rm e} (P(z)) \, dz \ + \ O_d(H^{-1})
\end{equation}
where 
$$
{\mathcal F} \ = \ \bigl\{t\in R_{s_{*}} : \eqref{F(n-2)-condition} \ {\rm holds \ for} \
z = t, \ \  |P'(t)| = p^{s_{*}} \ \ {\rm and} \ |P''(t)/2| = p^{2s_{*} -1} \bigr\}.
$$

In what follows it will be important to first make the reduction to the case $s_{*}\ge 2$
(recall we have reduced
to the case $s_{*}\ge 1$). Suppose that $s_{*} = 1$. Then 
${\mathcal F} \subset {\mathbb Z}/p{\mathbb Z}$.
Recall that
$P(z) = c_d z^d + \cdots + c_1 z$ and $\lambda = \max |c_j| \ge p$. Also recall
$n$ was chosen so that $|c_n| = \lambda$
and $|c_j| < \lambda$ for all $j>n$. Hence $|P^{(n)}(z)/n!| \equiv |c_n| = \lambda$
and since $|P^{(n)}(z)/n!| \le p^{n s_{*} - (n-1)} = p$ for $z\in {\mathcal F}$, 
we see that $\lambda = |c_n| = p$. 

Therefore if ${\mathcal F}\not= \emptyset$, then
for any $t\in {\mathcal F}$, we have $H = |P'(t)| = p$, $J_P(t) = |P''(t)/2| = p$ and 
$|P^{(k)}(t)/k!| \le p$ for all $k\ge 1$. Hence \eqref{P-derivatives} holds in this case
(when ${\mathcal F} \not= \emptyset$).

We claim that
\begin{equation}\label{p-F}
\# ({\mathbb Z}/p{\mathbb Z} \setminus {\mathcal F}) \  \lesssim_d \ 1.
\end{equation}
Let 
$$
S_n \ = \ \bigl\{t \in {\mathbb Z}/p{\mathbb Z}: |P^{(n)}(t)/n!|^{1/n} < |P^{(n-1)}(t)/(n-1)!|^{1/(n-1)} \  \bigr\}
$$
and inductively, for $2\le k \le n$, define
$$
S_k \ =  \bigl\{ t \in S_{k+1} :  |P^{(k)}(t)/k!|^{1/k} < |P^{(k-1)}(t)/(k-1)!|^{1/(k-1)} \ \bigr\}.
$$
Hence $S_3$ is precisely the set of $t \in {\mathbb Z}/p{\mathbb Z}$ where the string of inequalities in \eqref{F(n-2)-condition}
holds. 

We claim that $S_n = \{t\in {\mathbb Z}/p{\mathbb Z} : |P^{(n-1)}(t)/(n-1)!| = p \}$ and more generally,
$$
S_k \ = \ \bigl\{ t\in {\mathbb Z}/p{\mathbb Z} : |P^{(j)}(t)/j!| = p, \ {\rm for \ all} \ k\le j \le n \bigr\}.
$$
In fact since $|P^{(n)}(t)/n!| \equiv |c_n| = p$, then $t \in S_n$ precisely when 
$$
p^{n/(n-1)} \ < \ |P^{(n-1)}(t)/(n-1)!| \ \ (\le p)
$$
or presisely when $|P^{(n-1)}(t)/(n-1)!| = p$. By induction, the same is true for all $S_k$.
Therefore in this case, $t \in {\mathcal F}$ precisely when
$$
|P^{(k)}(t)/k!| \ = \ p \ \ {\rm for \ all} \ 1\le k \le n.
$$
Set $b_j = p c_j$ so that $|b_j| = p^{-1} |c_j| \le 1$ and $Q(t) = p P(t) \in {\mathbb Z}_p[X]$. Hence 
$t \in {\mathbb Z}/p{\mathbb Z} \setminus {\mathcal F}$ precisely when 
$$
Q^{(k)}(t)/k! \ \equiv \ 0 \ \ {\rm mod} \ \ p, \ \ \ {\rm for \ some} \ 1\le k \le n.
$$
Since these are polynomial equations over a field, this failure can only happen $O_d(1)$ times.
This proves the claim \eqref{p-F}.

Hence if $Q(t) = p P(t)$,
$$
I^{s_{*}} \ = \sum_{t\in {\mathcal F}} \int_{B_{p^{-1}}(t)} {\rm e} (P(z)) \, dz \ + O_d(H^{-1}) \ = \ 
p^{-1} \sum_{t\in {\mathcal F}} e^{2\pi i Q(t)/p} \ + \ O_d(H^{-1})
$$
$$
= \ p^{-1} \sum_{t=0}^{p-1} e^{2\pi i Q(t)/p} \ + \ O_d(H^{-1})
$$
since $H = p^{s_{*}} = p$ in this case. Since $p$ does not 
divide the $n$th coefficient of $Q$, we see that we are in the $\epsilon = 1$ senario
of Theorem \ref{H-main} when ${\mathcal F} \not= \emptyset$.

So we may assume $s_{*} \ge 2$.

We now decompose ${\mathcal F} = {\mathcal F}' \cup {\mathcal F}''$ where
${\mathcal F}' = \{ t\in {\mathcal F}: |t| \ge p^{-s_{*}+2}\}$.

\section{Establishing $\# {\mathcal F}' \lesssim_d 1$}\label{establishing}  

We reduce to the case of equality in \eqref{F(n-2)-conditionv2} for $t\in {\mathcal F}'$; 
that is, 
$|P^{(k)}(t)/k!| = p^{ks_{*} -(k-1)}$. To do this, we decompose 
${\mathcal F}' = G_3 \cup \cdots \cup G_n \cup G'$ where
$$
G_3 \ := \ \{t\in {\mathcal F}' : |P^{'''}(t)/3!| \le p^{3s_{*} - 3} \} 
$$
and inductively, we define
$$
G_k \ := \ \{t \in {\mathcal F}' \setminus \cup_{\ell=1}^{k-1} G_{\ell} 
: |P^{(k)}(t)/k!| \le p^{ks_{*} - k} \}
$$
for $3\le k \le n$. We note that if $t\in {\mathcal F}'\setminus G_3$, 
then $|P'''(t)/3!| = p^{3 s_{*} - 2}$ and
more generally, if $t \in G_k$, then
\begin{equation}\label{Gk}
|P^{(\ell)}(t)/\ell!| \ = \ p^{\ell s_{*} - (\ell -1)} \ {\rm for} \ 1\le \ell \le k-1, \ \ {\rm and} \ \ 
|P^{(k)}(t)/k!| \le p^{k s_{*} - k}.
\end{equation}
Finally if $t \in G'$, then we have $|P^{(k)}(t)/k!| = p^{k s_{*} - (k-1)}$ for all $1\le k \le n$.

For $t\in G_3$, we apply Lemma \ref{hensel-L} with $L = n-1$ to $\phi = P'$ and $t_0 = t$.
We have $\delta = |P'(t) P''(t)^{-1} P^{(n)}(t)^{-1}|$ and so
$$
\lambda \delta \  = \ p^{-s_{*} + 1} \ < \ 1
$$ 
since $s_{*} \ge 2$. Here we used the fact that $|P'(t)| = p^{s_{*}}$ and $|P''(t)/2| = p^{2s_{*}-1}$
when $t\in {\mathcal F}'$.

Next we need to verify\footnote{For the rest of this section, we again drop factorials for notational
convenience.} $\delta_1 < 1$. Since $|P'''(t)| \le p^{3s_{*} - 3}$ for $t\in G_3$, we have  
$$
\delta_1 = |\phi(t_0) \phi'(t_0)^{-2} \phi''(t_0)| \ = \ |P'(t) P''(t)^{-2} P'''(t)| 
\le p^{s_{*} - 4s_{*}+2 + 3 s_{*} -3}
= p^{-1}
$$
and so $\delta_1 < 1$. Also for $2\le \ell \le n-1$, we need to verify that $\delta_{\ell}\le 1$
where $\delta_{\ell} = |\phi(t_0) \phi'(t_0)^{-1} \phi^{(\ell+1)}(t_0) \phi^{(\ell)}(t_0)^{-1}|$. But by \eqref{F(n-2)-conditionv3},
$$
|\phi^{(\ell+1)}(t_0)\phi^{(\ell)}(t_0)^{-1}| = |P^{(\ell+2)}(t) P^{(\ell +1)}(t)^{-1}| < p^{s_{*}}
$$
and so
$$
\delta_{\ell} < p^{s_{*}} |P'(t) P''(t)^{-1}| \ = \ p^{s_{*} + s_{*} - 2s_{*} +1} \ = \ p,
$$
implying $\delta_{\ell} \le 1$. Hence there exists a zero $t_{*} \in {\mathbb Q}_p$ of $P^{'}$
with $|t - t_{*}| \le |P'(t) P''(t)^{-1}| = p^{-s_{*}+1}$.

Since distinct elements $t, t' \in {\mathcal F}'$ are separated $|t - t'| \ge p^{-s_{*}+2}$,
we see that
\begin{equation}\label{G3}
\# G_3  \ \lesssim_d \ 1.
\end{equation} 

For general $3\le k \le n$ and $t\in G_k$, we apply Lemma \ref{hensel-L} with $L= n - k+2$
to $\phi = P^{(k-2)}$ and $t_0 = t$. 

We have $\delta = |P^{(k-2)}(t) P^{(k-1)}(t)^{-1} P^{(n)}(t)^{-1}|$ and so
$$
\lambda \delta \  = \  p^{-s_{*} + 1} \ < \ 1
$$ 
since $s_{*}\ge 2$.

Next we need to verify $\delta_1 < 1$. Since $|P^{(k)}(t)| \le p^{ks_{*} - k}$ for $t\in G_k$, 
we have  
$$
\delta_1 \ = \ |\phi(t_0) \phi'(t_0)^{-2} \phi''(t_0)| \ = \ |P^{(k-2)}(t) P^{(k-1)}(t)^{-2} P^{(k)}(t)|
$$
$$ \le \ 
p^{(k-2)s_{*} - 2(k-1)s_{*}+ k s_{*} + 2(k-1) - k - (k-3)}
\ = \ p^{-1}
$$
by \eqref{F(n-2)-conditionv2}
and so $\delta_1 < 1$. Also for $2\le \ell \le n-1$, we need to verify that $\delta_{\ell}\le 1$
where $\delta_{\ell} = |\phi(t_0) \phi'(t_0)^{-1} \phi^{(\ell+1)}(t_0) \phi^{(\ell)}(t_0)^{-1}|$. But from 
\eqref{F(n-2)-conditionv3},
$$
|\phi^{(\ell+1)}(t_0)\phi^{(\ell)}(t_0)^{-1}| = |P^{(k+\ell-1)}(t) P^{(k+\ell -2)}(t)^{-1}| < p^{s_{*}}
$$
and so
$$
\delta_{\ell} < p^{s_{*}} |P^{(k-2)}(t) P^{(k-1)}(t)^{-1}| \ = \ p^{s_{*} + (k-2)s_{*} - (k-1)s_{*} +1} \ = \ p,
$$
implying $\delta_{\ell} \le 1$. Hence there exists a zero $t_{*} \in {\mathbb Q}_p$ of $P^{(k-2)}$
with $|t - t_{*}| \le |P^{(k-2)}(t) P^{(k-1)}(t)^{-1}| = p^{-s_{*}+1}$. 

Since distinct elements $t, t' \in {\mathcal F}'$ are separated $|t - t'| \ge p^{-s_{*}+2}$,
we see that
\begin{equation}\label{Gk}
\# G_k \ \lesssim_d \ 1.
\end{equation}

Finally we are left to treat $j\in G'$ where $|P^{(k)}(t)/k!| = p^{k s_{*}- (k-1)}$ for
all $1\le k \le n$. Here we simply employ Lemma \ref{hensel-L} with $L=1$ to $\phi = P^{(n-1)}$ and $t_0 = t$.
Note that
$$
\lambda \delta \ = \ \lambda |P^{(n-1)}(t) P^{(n)}(t)^{-2}|  \ = \  |P^{(n-1)}(t) P^{(n)}(t)^{-1}| \ = \
p^{-s_{*} +1} \ < \ 1
$$
since $s_{*}\ge 2$.
Hence there exists a zero $t_{*} \in {\mathbb Q}_p$ of $P^{(n-1)}$
with $|t - t_{*}| \le |P^{(n-1)}(t) P^{(n)}(t)^{-1}| = p^{-s_{*}+1}$.

Since distinct elements $t, t' \in {\mathcal F}'$ are separated $|t - t'| \ge p^{-s_{*}+2}$,
we see that
\begin{equation}\label{G'}
\# G' \ \lesssim_d \ 1.
\end{equation}

The bound $\# {\mathcal F}' \lesssim_d 1$ follows from \eqref{G3}, \eqref{Gk} and \eqref{G'}.

Hence from \eqref{I-calF},
\begin{equation}\label{I-calF''}
I^{s_{*}} \ = \ \sum_{t\in {\mathcal F}''} \ \int_{B_{p^{-s_{*}}}(t)} {\rm e} (P(z)) \, dz \ + \ O_d(H^{-1})
\end{equation}
since $H \le p^{s_{*}}$.

\section{The final push}\label{push}

It remains to treat the sum 
$$
\sum_{t\in {\mathcal F}''} \int_{B_{p^{-s_{*}}}(t)} {\rm e} (P(z)) \, dz \ =: \ \sum_{t\in {\mathcal F}''} I(t).
$$ 
Changing variables, we have
$$
I(t) \ = \ p^{-s_{*}} {\rm e} (P(t)) \int_{u \in {\mathbb Z}_p} 
{\rm e} \bigl( \sum_{k\ge 1} [P^{(k)}(t)/k!] (p^{s_{*}} u)^k \bigr) \, du \ = \ p^{-s_{*}} {\rm e} (P(t))
$$
since each term $[P^{(k)}(t)/k!]  (p^{s_{*}} u)^k \in {\mathbb Z}_p$ which holds since
$|P^{(k)}(t)/k!| p^{-k s_{*}} \le [J(t) p^{-s_{*}}]^k \le 1$.

Since for $t\in {\mathcal F}'', \ t = p^{s_{*} - 1} w$ for some $w \in {\mathbb Z}/p{\mathbb Z}$, 
we can view ${\mathcal F}''$ as a subset of ${\mathbb Z}/p{\mathbb Z}$ with $w \in {\mathcal F}''$ if and only if
$$
|P^{(k)}(p^{s_{*}-1} w)/k!|^{1/k} < |P^{(k-1)}(p^{s_{*}-1}w)/(k-1)!|^{1/(k-1)}
\ \ {\rm for} \ 3\le k\le n
$$
with $|P''(p^{s_{*}-1}w)/2| = p^{2s_{*} -1}$ and $|P'(p^{s_{*}-1} w)| = p^{s_{*}}$.

Recall $P(t) = c_d t^d + \cdots + c_1 t$ with $\lambda = \max |c_j| = |c_n|$. 
By \eqref{F(n-2)-conditionv2}, we have
\begin{equation}\label{ksu}
|P^{(k)}(p^{s_{*}-1}w)/k!| \ \le \ p^{ks_{*} - (k-1)} \ \ {\rm for \ all} \ \ 1\le k \le d
\end{equation}
and this verifies \eqref{P-derivatives} for points $p^{s_{*}-1}w \in {\mathcal F}''$.
Strictly speaking \eqref{F(n-2)-conditionv2} implies that \eqref{ksu} holds for $1\le k \le n$.
However for $k>n$,
$$
|P^{(k)}(p^{s_{*}-1}w)/k!| \ \le \ \lambda \ = \ |c_n| \ \equiv \ |P^{(n)}(p^{s_{*}-1}w)/n!| \ \le \
p^{n s_{*} - (n-1)} \ \le \ p^{ks - (k-1)}
$$
and this puts us in position to see that \eqref{I-single-refined} holds 
when ${\mathcal F}'' \not= \emptyset$.

Our final analysis
relies on 2 key claims.

{\bf Claim 1}: For $t = p^{s_{*}-1} w \in {\mathcal F}''$, we have $|c_j| \le p^{j s_{*} - (j-1)}$ for each $1\le j \le d$
with equality for some $j\ge 2$. 

To prove the claim, we assume ${\mathcal F}'' \not= \emptyset$ and fix $p^{r-1}u \in {\mathcal F}''$.
Also we may assume $j\le n$ since $|c_j| \le |c_n|$ for $j\ge n$.

First suppose
that $|c_n| \ge p^{ns_{*} - n+2}$. We have 
$$
p^{ns_{*} - n+2} \ \le \ |c_n| \ = \ |P^{(n)}(p^{s_{*}-1}u)/n!| \ \le \ p^{n s_{*} - (n-1)}
$$
which is a contradiction. Hence $|c_n| \le p^{n s_{*} - (n-1)}$. 

Now suppose by induction that $|c_j| \le p^{j s_{*} - (j-1)}$ for all $k<j\le n$ and by contradiction,
$|c_k| \ge p^{ks_{*} -k+2}$. We have
$$
P^{(k)}(p^{s_{*}-1}u)/k! \ = \ c_k + \sum_{\ell=1}^{n-k} [c_{k+\ell}/k!\ell!] (p^{s_{*}-1}u)^{\ell} \ + \ E
$$
where 
$$
|E| \ \le \ \lambda p^{-(s_{*}-1)(n-k+1)} \ = \ |c_n| p^{-(s_{*}-1)(n-k+1)} \ \le \ p^{n s_{*} - (n-1)} p^{-(s_{*}-1)(n-k+1)},
$$ 
implying $|E| \le p^{k s_{*} - (k-1)} p^{-s_{*} + 2} \le p^{k s_{*} - (k-1)} < |c_k|$ since $s_{*} \ge 2$.

Also for $1\le \ell \le n-k$,
$|c_{k+\ell}| \le p^{(k+\ell)r - (k+\ell-1)}$, implying each term in the above sum has the bound
$$
|c_{k+\ell}| p^{-\ell (s_{*}-1)} \ \le \ p^{(k+\ell)s_{*} - (k+\ell -1)} p^{-\ell(s_{*}-1)} \ = \ p^{ks_{*} - (k-1)} \ < \ |c_k|.
$$
Hence $|P^{(k)}(p^{s_{*}-1}u)/k!| = |c_k| \ge p^{k s_{*} -k+2}$, contradicting \eqref{ksu}.
This shows that $|c_j| \le p^{js_{*} - (j-1)}$ for all $1\le j \le n$.

Since $p^{s_{*}-1}u \in {\mathcal F}''$, we have
\begin{equation}\label{2s}
p^{2s_{*} -1} \ = \ |P^{''}(p^{s_{*}-1}u)/2| \ = \ \bigl|c_2 + \sum_{\ell=1}^{n-2} [c_{2+\ell}/2 \ell!] 
(p^{s_{*}-1}u)^{\ell} \ + \ E \bigr|
\end{equation}
where $|E| \le \lambda p^{-(n-1)(s_{*}-1)} = |c_n| \, p^{-(n-1)(s_{*}-1)} \le p^{s_{*}}$.

Let us now see that \eqref{2s} implies that
equality $|c_j| = p^{js_{*} - (j-1)}$ holds for some $2\le j\le n$. In fact if $|c_j| < p^{js_{*} - (j-1)}$ for all $2\le j \le n$,
then 
$$
|E| \ \le \ |c_n| \, p^{-(n-1)(s_{*} -1)} \ < \ p^{s_{*}} \ \le \ p^{2s_{*} -1}
$$ 
since $s_{*}\ge 1$. Also 
each term in the sum in \eqref{2s} has the bound 
$$
|c_{2+\ell}| \, p^{-\ell(s_{*}-1)} \ < \ p^{(2+\ell)s_{*} - (\ell +1)} p^{-\ell(s_{*}-1)} \ = \ p^{2s_{*}-1},
$$
which contradicts \eqref{2s}. The completes the proof of Claim 1.

{\bf Claim 2}: We have $\# S \lesssim_d 1$ where $S$ consists of those $w \in {\mathbb Z}/p{\mathbb Z}$ where 
$p^{s_{*}-1} w \notin {\mathcal F}''$.

The proof of this claim is identical to the proof of \eqref{p-F}. The elements 
$w \in {\mathbb Z}/p{\mathbb Z}$ where $p^{s_{*}-1}w \in {\mathcal F}''$ are precisely
those elements which solve the equations
$$
|P^{(k)}(p^{s_{*}-1}w)/k!| \ = \ \lambda \ \ {\rm for \ all} \ 1\le k \le n.
$$
Set $b_j = \lambda c_j$ so that $|b_j| = \lambda^{-1} |c_j| \le 1$ and $Q(t) = \lambda P(t) \in {\mathbb Z}_p[X]$. Hence 
$w \in S$ precisely when 
$$
Q^{(k)}(p^{s_{*}-1}w)/k! \ \equiv \ 0 \ \ {\rm mod} \ \ p \ \ \ {\rm for \ some} \ 1\le k \le n.
$$
Since these are polynomial equations over a field, there are at most $O_d(1)$ solutions.
This proves Claim 2.

When ${\mathcal F}'' \not= \emptyset$, we can use Claim 1 and Claim 2 to 
write the sum $\sum_{t\in {\mathcal F}''} I(t)$ as
$$
p^{-s_{*}} \, \sum_{w=0}^{p-1} {\rm e} (P(p^{s_{*}-1}w)) \ + \ O_d( p^{-s_{*}})
\ = \ p^{-s_{*}} \sum_{w=0}^{p-1} e^{2\pi i Q(w)/p} \ + \ O_d(H^{-1})
$$ 
where $Q(w) = e_1 w + e_2 w^2 + \cdots + e_d w^d$ and $e_j = p^{j s_{*} - (j-1)} c_j$. Hence $|e_j| \le 1$
with equality for some $2\le j \le n$. 

From \eqref{I-reduction} and \eqref{I-calF''}, we can therefore write
\begin{equation}\label{I-final}
I(H) \ = \ \epsilon \, p^{-s_{*}} \sum_{w=0}^{p-1} e^{2\pi i Q(w)/p} \ + \ O_d(H^{-1})
\end{equation}
where not all coefficients of $Q$ are equal to 0 mod $p$. Here $\epsilon \in \{0,1\}$
and $\epsilon =1$ when ${\mathcal F}'' \not= \emptyset$. In this case, we've seen
$H = |P'(z)| = p^{s_{*}}, J_P(z) = |P''(z)/2|^{1/2} = p^{s_{*} -1/2}$ and \eqref{P-derivatives}
holds for any $z \in {\mathcal F}''$. 

This completes the proof Theorem \ref{H-main}.



\section{Proof of Proposition \ref{sublevel-structure-basic}}\label{Prop-structure}

We conclude with proofs of Proposition \ref{sublevel-structure-basic} and Lemma \ref{hensel-L}.
Proposition \ref{sublevel-structure-basic} lies at the heart of the key structural statement about sublevel sets
and Lemma \ref{hensel-L} is what we need to establish this structural statement.

Recall that for Proposition \ref{sublevel-structure-basic}, we have a polynomial $Q \in {\mathbb Q}_p[X]$
and a point $z \in {\mathbb Z}_p$ such that $|Q(z)| \le p^{-L}$ for some $L\ge 1$ and $|Q^{(k)}(z)/k!| \ge 1$
for some $k\ge 1$. And we are looking for a zero $z_{*} \in {\mathbb Q}_p$ of some derivative 
of $Q$ such that $|z - z_{*}| \le p^{-L/k}$.

If $Q(w) = c_0 + c_1 w + \cdots + c_d w^d$, set $\lambda = \max_j |c_j|$. If $\lambda > 1$,
let $n$ be the largest integer such that $|c_n| = \lambda$ and $|c_j| < \lambda$ for all $j>n$.
Then 
$$
Q^{(n)}(w)/n! \ = \ c_n \ + \ (n+1) c_{n+1} w \ + \ \cdots \ + \  d(d-1)\cdots(d-n+1) c_d w^{d-n},
$$
implying $|Q^{(n)}(w)/n!| \equiv |c_n| = \lambda$. Furthermore if $k>n$,
$$
1 \ \le \ |Q^{(k)}(z)/k!| \ < \ \lambda \ = \ |Q^{(n)}(z)/n!|
$$ 
which implies
$|Q^{(k)}(z)/k!|^{1/k} \le |Q^{(n)}(z)/n!|^{1/n}$ and in particular,
\begin{equation}\label{k>n}
|p^{-L} Q^{(k)}(z)/k!|^{1/k} \ \le \ |p^{-L} Q^{(n)}(z)/n!|^{1/n}
\end{equation}
since $L\ge 1$.
We set $n_{*} := n$ if $\lambda > 1$ and
$n_{*} = k$ if $\lambda \le 1$. We also set $\lambda_{+} := \max(1, \lambda)$.
Note that $\lambda_{+} \le |Q^{(n_{*})}(z)/n_{*}!|$.

We consider a number of cases.\footnote{Again we suppress factorials for notational convenience.}

{\bf Case 0}: Suppose
\begin{equation}\label{n}
\bigl| p^{-L} Q^{(n_{*}-1)}(z)\bigr|^{1/(n_{*}-1)} \ \le \ \bigl| p^{-L} Q^{(n_{*})}(z)\bigr|^{1/n_{*}}.
\end{equation}

In this case we apply Lemma \ref{hensel-L} with $L=1$ to $\phi = Q^{(n_{*}-1)}$ and $t_0 = z$.
We have 
$$
\lambda_{+} \delta \ = \ \lambda_{+} |Q^{(n_{*}-1)}(z) Q^{(n_{*})}(z)^{-2}|
$$ 
and so $\lambda_{+} \delta \le |Q^{(n_{*}-1)}(z) Q^{(n_{*})}(z)^{-1} |$
since $\lambda_{+} \le |Q^{(n_{*})}(z)|$. Hence by \eqref{k>n} and \eqref{n},
$$
\lambda_{+} \delta  \ \le \
|p^{-L} Q^{(n_{*})}(z)|^{-1/n_{*}} \ \le \ |p^{-L}Q^{(k)}(z)|^{-1/k} \ \le \ p^{-L/k} \ < \ 1
$$
since $|Q^{(k)}(z)| \ge 1$ and $L\ge 1$. Therefore there exists a zero $z_{*}$ of $Q^{(n_{*}-1)}$ such that
\begin{equation}\label{0}
|z - z_{*}| \ \le \ |Q^{(n_{*}-1)}(z) Q^{(n_{*})}(z)^{-1}| \ \le \ p^{-L/k} .
\end{equation}

{\bf General case j}: Here we suppose $1\le j \le n_{*}$,
\begin{equation}\label{n-j}
\bigl| p^{-L} Q^{(n_{*})}(z)\bigr|^{1/n_{*}} \ < \ \cdots \ < \ |p^{-L} Q^{(n_{*}-j)}(z)|^{1/(n_{*}-j)} 
\end{equation}
and
\begin{equation}\label{j}
 \ \ \bigl| p^{-L} Q^{(n_{*} - j-1)}(z)\bigr|^{1/(n_{*}-j-1)} \ \le \ 
\bigl| p^{-L} Q^{(n_{*} - j)}(z)\bigr|^{1/(n_{*}-j)}. 
\end{equation}

In this case we apply Lemma \ref{hensel-L} with $L=j+1$ to $\phi = Q^{(n_{*}-j-1)}$ and $t_0 = z$.
We have 
$$
\lambda_{+} \delta \ = \ \lambda_{+} |Q^{(n_{*}-j-1)}(z) Q^{(n_{*}-j)}(z)^{-1} Q^{(n_{*})}(z)^{-1}(z)|
$$ 
and so $\lambda_{+} \delta \le |Q^{(n_{*}-j-1)}(z) Q^{(n_{*}-j)}(z)^{-1} |$
since $\lambda_{+} \le |Q^{(n_{*})}(z)|$. Hence by \eqref{k>n} and \eqref{j},
$$
\lambda_{+} \delta  \ \le \
|p^{-L} Q^{(n_{*}-j)}(z)|^{-1/(n_{*}-j)} \ \le \ |p^{-L}Q^{(k)}(z)|^{-1/k} \ \le \ p^{-L/k} \ < \ 1
$$
since $|Q^{(k)}(z)| \ge 1$ and $L\ge 1$. Furthermore,
$$
\delta_1 \ = \ |Q^{(n_{*} - j -1)}(z) Q^{(n_{*}-j)}(z)^{-2} Q^{(n_{*}-j+1)}(z)|
$$
and so by \eqref{n-j} and \eqref{j}, 
we have $\delta_1 < 1$. Hence when $j=1$, we conclude that
there exists a zero $z_{*}$ of $Q^{(n_{*}-2)}$ such that
\begin{equation}\label{j=1}
|z - z_{*}| \ \le \ |Q^{(n_{*}-2)}(z) Q^{(n_{*}-1)}(z)^{-1}| \ \le \ p^{-L/k} .
\end{equation}

When $j\ge 2$, we also need to verify that $\delta_{\ell} \le 1$ for $2\le \ell \le j$ where
$$
\delta_{\ell} = |Q^{(n_{*}-j-1)}(z) Q^{(n_{*}-j)}(z)^{-1} Q^{(n_{*}-j+\ell)}(z) Q^{(n_{*}-j+\ell-1)}(z)^{-1}|
$$
so that
$$
\delta_{\ell} \ \le \ |p^{-L} Q^{(n_{*}-j)}(z)|^{-1/(n_{*}-j)} \, 
|Q^{(n_{*}-j+\ell)}(z) Q^{(n_{*}-j+\ell-1)}(z)^{-1}|.
$$
Consider the following
property:
$$
|Q^{(n_{*}-j+\ell)}(z) Q^{(n_{*}-j+\ell-1)}(z)^{-1}| \ \le \ |p^{-L} Q^{(n_{*}-j)}(z)|^{1/(n_{*}-j)}
\eqno{(P_{\ell})}
$$
so that 
$$
\delta_{\ell} \ \le \ 1 \ \ {\rm when} \ \ (P_{\ell}) \ {\rm holds}.
$$ 
Therefore {\bf if} $(P_{\ell})$ holds for every $2\le \ell \le j$, then we can deduce that there
exists a zero $z_{*}$ of $Q^{(n_{*}-j-1)}$ such that
\begin{equation}\label{j=1}
|z - z_{*}| \ \le \ |Q^{(n_{*}-j-1)}(z) Q^{(n_{*}-j)}(z)^{-1}| \ \le \ p^{-L/k}
\end{equation}
since by \eqref{j},
$$
|Q^{(n_{*}-j-1)}(z) Q^{(n_{*}-j)}(z)^{-1}| \ \le \ |p^{-L} Q^{(n_{*}-j)}(z)|^{-1/(n_{*}-j)} \ \le \
|p^{-L} Q^{(k)}(z)|^{-1/k}.
$$

It remains to consider the subcase where one of the properties $(P_{\ell})$ does {\bf not} hold.

{\bf Subcase j (0)}: Suppose that $(P_j)$ does not hold so that
\begin{equation}\label{P-j-not}
|Q^{(n_{*}-1)}(z) Q^{(n_{*})}(z)^{-1}| \ < \ 
|p^{-L} Q^{(n_{*}-j)}(z)|^{-1/(n_{*}-j)}.
\end{equation}
 In this case, we apply Lemma \ref{hensel-L}
with $L=1$ to $\phi = Q^{(n_{*}-1)}$. Then by \eqref{P-j-not},
$$
\lambda_{+} \delta \le \ |Q^{(n_{*}-1)}(z) Q^{(n_{*})}(z)^{-1}| \ < \  
|p^{-L} Q^{(n_{*}-j)}(z)|^{-1/(n_{*}-j)}
$$
and by \eqref{k>n} and \eqref{n-j}, we see that
$$
\lambda_{+} \delta \ < \ |p^{-L} Q^{(k)}(z)|^{-1/k} \ \le \ p^{-L/k} < 1.
$$
Hence there exists a zero $z_{*}$ of $Q^{(n_{*}-1)}$ such that
$$
|z - z_{*}| \ \le \ |Q^{(n_{*}-1)}(z) Q^{(n_{*})}(z)^{-1}| \ < \ 
|p^{-L} Q^{(n_{*}-j)}(z)|^{-1/(n_{*}-j)}.
$$
The last inequality is \eqref{P-j-not}. Hence
\begin{equation}\label{sub-0}
|z - z_{*}| \ \le \ |p^{-L} Q^{(k))}(z)|^{-1/k} \ \le \ p^{-L/k}.
\end{equation}

{\bf General subcase j}: For $r\ge 2$, suppose that property $(P_{\ell})$ holds for
$r+1 \le\ell \le j$ and $(P_r)$ does not hold so that
\begin{equation}\label{P-r-not}
|Q^{(n_{*}-j+r -1)}(z) Q^{(n_{*}-j+r)}(z)^{-1}| \ < \ 
|p^{-L} Q^{(n_{*}-j)}(z)|^{-1/(n_{*}-j)}.
\end{equation}
 In this case, we apply Lemma \ref{hensel-L}
with $L=r-1$ to $\phi = Q^{(n_{*}-j+r-1)}$. Then by \eqref{P-r-not},
$$
\lambda_{+} \delta \le \ |Q^{(n_{*}-j+r-1)}(z) Q^{(n_{*}-j+r)}(z)^{-1}| \ < \  
|p^{-L} Q^{(n_{*}-j)}(z)|^{-1/(n_{*}-j)}
$$
and by \eqref{k>n} and \eqref{n-j}, we see that
$$
\lambda_{+} \delta \ < \ |p^{-L} Q^{(k)}(z)|^{-1/k} \ \le \ p^{-L/k} < 1.
$$
Furthermore for $1\le m \le  r-2$,
$$
\delta_m = |Q^{(n_{*}-j+r-1)}(z) Q^{(n_{*}-j+r)}(z)^{-1} Q^{(n_{*} -j +r+m)}(z) Q^{(n_{*}-j+r+m-1)}(z)^{-1}|
$$
and so by \eqref{P-r-not} and since $(P_{\ell})$ holds for $r+1 \le \ell \le j$, we have
$\delta_m < 1$. Hence there exists a zero $z_{*}$ of $Q^{(n_{*}-j+r-1)}$ such that
$$
|z - z_{*}| \ \le \ |Q^{(n_{*}-j+r-1)}(z) Q^{(n_{*}-j+r)}(z)^{-1}| \ < \ 
|p^{-L} Q^{(n_{*}-j)}(z)|^{-1/(n_{*}-j)}.
$$
The last inequality is \eqref{P-r-not}. Hence
\begin{equation}\label{sub-r}
|z - z_{*}| \ \le \ |p^{-L} Q^{(k))}(z)|^{-1/k} \ \le \ p^{-L/k}.
\end{equation}

This completes the proof of Proposition \ref{sublevel-structure-basic}.


\section{Proof of Lemma \ref{hensel-L}}\label{hensel-proof}

Finally we end with the proof of Lemma \ref{hensel-L}.

We only give the proof when $L\ge 2$. The proof when $L=1$ is easier. Recall that
$\phi(t) = c_0 + c_1 t + \cdots + c_d t^d \in {\mathbb Q}_p[X]$ and 
$\lambda_{+} := \max(1, \lambda)$ where $\lambda = \max_j |c_j|$.

We define a sequence $\{t_n\}$ recursively by
\begin{equation}\label{recursive}
t_n \ = \ t_{n-1} - \phi(t_{n-1}) \phi'(t_{n-1})^{-1}
\end{equation}
and recall $\delta_1 = |\phi(t_0) \phi'(t_0)^{-2}\phi''(t_0)/2|$ and our hypothesis
$\delta_1 < 1$. We make the
following claim.

{\bf Claim:} \ For every $n\ge 1$,

$(1)_n \ \ |t_n - t_{n-1}| \ \le \ |\phi(t_{0}) \phi'(t_{0})^{-1}| \, \delta_1^{2^{n-1} - 1}, $

$(2)_n \ \ |\phi(t_{n-1})| \ \le \ |\phi(t_0)| \delta_1^{2^{n-1}-1}, $


$(3)_n$ \ for $1\le k \le L$, $|\phi^{(k)}(t_{n-1})/k!| = |\phi^{(k)}(t_0)/k!|$ .

Note the claim implies that for every $n\ge m$,
$$
|t_n - t_m|  \le  \max_{m<j\le n} |t_j - t_{j-1}|  \le  |\phi(t_0)\phi'(t_0)^{-1}| \max_{m<j\le n} \delta_1^{2^{j-1}-1}  
\to 0 \ \ {\rm as} \ m,n \to \infty.
$$
Hence $\{t_n\} \subseteq {\mathbb Q}_{p}$ is Cauchy and so $t_n \to t$ for some $t\in {\mathbb Q}_p$.
Furthermore $|t-t_0| \le |\phi(t_0) \phi'(t_0)^{-1}|$ and $(2)_n$ implies $\phi(t) = 0$. Therefore
it remains to prove the claim.

We now proceed with the proof of the claim.
The claim is true for $n=1$ and so suppose $(1)_j, (2)_j$ and $(3)_j$ holds for all $1\le j \le n$.
Note that $(3)_n$ implies $\phi'(t_{n-1}) \not= 0$ and so $z_n$ is well-defined. 
We begin with proving $(3)_{n+1}$. For $1\le k \le L$, we see
\begin{equation}\label{TS-II}
\phi^{(k)}(t_n) \ = \ \phi^{(k)}(t_{n-1}) + \sum_{j=1}^{L-k} \frac{\phi^{(k+j)}(t_{n-1})}{j!} (t_n - t_{n-1})^j \ + \ R_{L-k,k}
\end{equation}
where by $(1)_n$,
$$
|R_{L-k,k}| \ \le \ \lambda_{+} |t_n - t_{n-1}|^{L-k+1} \ \le \ \lambda_{+} |\phi(t_0)\phi'(t_0)^{-1}| \delta_1^{2^{n-1}-1} 
|t_n - t_{n-1}|^{L-k}
$$
\begin{equation}\label{remainder-II}
\ = \ \lambda_{+} \delta (\delta_1^{2^{n-1}-1}) |\phi^{(L)}(t_0)/L!|  |t_n - t_{n-1}|^{L-k} .
\end{equation}
Furthermore by $(1)_n$ and $(3)_n$, for each $1\le j \le L-k$, 
$$
|\phi^{(k+j)}(t_{n-1})/j! \, (t_n - t_{n-1})^j | \ \le \ \delta_1^{j(2^{n-1}-1)} |\phi(t_0)\phi'(t_0)^{-1}|^j 
|\phi^{(k+j)}(t_0)/(k+j)!| .
$$
By the definition of the $\delta_{\ell}, \, 1 \le \ell \le L-1$ in the statement of Lemma \ref{hensel-L}, we have
$$
|(\phi(t_0)\phi'(t_0)^{-1})^j \phi^{(k+j)}(t_0)/(k+j)!| \ = \ |(\delta_k \cdots \delta_{k+j-1}) \phi^{(k)}(t_0)/k!| \ \le \
|\phi^{(k)}(t_0)/k!|
$$
for $1\le j \le L-k$ and hence
\begin{equation}\label{k+j}
|\phi^{(k+j)}(t_{n-1})/(k+j)! \, (t_n - t_{n-1})^j | \ \le \ \delta_1^{2^{n-1}-1}  |\phi^{(k)}(t_0)/k!| .
\end{equation}

Plugging \eqref{remainder-II} and \eqref{k+j} into \eqref{TS-II}, using the hypotheses
$\lambda_{+} \delta \le 1$ and $\delta_1 < 1$ and the nonarchimedean nature of the $p$-adic absolute value, we see that 
$$
|\phi^{(k)}(t_n)/k!| \ = \ |\phi^{(k)}(t_0)/k!| 
$$
for $1\le k \le L$, establishing $(3)_{n+1}$.

For $(2)_{n+1}$, we expand
$$
\phi(t_n)  \ = \ \phi(t_{n-1}) + \phi'(t_{n-1}) (t_n - t_{n-1}) + \sum_{j=2}^L \frac{\phi^{(j)}(t_{n-1})}{j!} \ (t_n - t_{n-1})^j \ + \
R_L
$$
and since
$\phi(t_{n-1}) + \phi'(t_{n-1}) (t_n - t_{n-1}) = 0$ by definition of $t_n$, we have
$$
\phi(t_n)  \ = \ \sum_{j=2}^L \frac{\phi^{(j)}(t_{n-1})}{j!} \ (t_n - t_{n-1})^j \ + \ R_L
$$
where 
\begin{equation}\label{remainder-IV}
|R_L| \ \le \ \lambda_{+} |t_n - t_{n-1}|^{L+1} \ \le \ \lambda_{+} \delta (\delta_1^{2^{n-1}-1})
 |\phi^{(L)}(t_0)/L!|  |t_n - t_{n-1}|^L 
\end{equation}
as before. For $2\le j \le L$, we have by $(1)_n$ and $(3)_n$,
$$
|\phi^{(j)}(t_{n-1})/j!| |t_n - t_{n-1}|^j \le  |\phi^{(j)}(t_0)/j!| |\phi(t_0) \phi'(t_0)^{-1}|^{j-1} \delta_1^{2^n - 2} |\phi(t_0)\phi'(t_0)^{-1}| .
$$
Proceeding as above, using the definition of the $\delta_{\ell}$'s, we have
$$
|\phi^{(j)}(t_0)/j!| |\phi(t_0) \phi'(t_0)^{-1}|^{j-1} \ = \ (\delta_2\cdots \delta_{j-1}) 
|\phi(t_0) \phi'(t_0)^{-1}\phi''(t_0) /2|,
$$
implying
$$
|\phi^{(j)}(t_{n-1})/j!| |t_n - t_{n-1}|^j \ \le  \delta_1 \delta_1^{2^n - 2} |\phi(t_0)|.
$$
We have a similar estimate for the right hand side of \eqref{remainder-IV} and so, altogether, we have
$$
|\phi(t_n)| \ \le \ |\phi(t_0)| \delta_1^{2^n -1},
$$
completing the proof of $(2)_{n+1}$.

Finally for $(1)_{n+1}$, we use $(2)_{n+1}$,$(3)_ {n+1}$ to see
$$
|t_{n+1} - t_n| \ = \ |\phi(t_n) \phi'(t_n)^{-1}| \ = |\phi(t_n)\phi'(t_0)^{-1}| \ \le \ \delta_1^{2^n -1} |\phi(t_0)|
$$
which establishes $(1)_{n+1}$,
completing the proof of the claim and hence the proof of Lemma \ref{hensel-L}.

\end{document}